\documentclass[DIV=14,letterpaper]{scrartcl}
\pdfoutput=1

\date{}

\usepackage{tikz}
\usetikzlibrary{calc}
\usetikzlibrary{matrix}
\usepackage{amsmath,amssymb,amsfonts,amsthm,subcaption}
\usepackage{color}                    

\usepackage{graphicx}
\usepackage{epstopdf}
\usepackage{enumerate}

\allowdisplaybreaks

\def\d{\delta}

\newcommand{\ud}{\mathrm{d}}

\newcommand{\ux}{\boldsymbol{x}}

\newcommand{\uu}{\boldsymbol{u}}
\newcommand{\uv}{\boldsymbol{v}}

\newcommand{\uD}{\mathbb{D}}

\newcommand{\bey}{\begin{eqnarray}}
\newcommand{\eey}{\end{eqnarray}}

\newcommand{\beq}{\begin{equation}}
\newcommand{\eeq}{\end{equation}}
\theoremstyle{plain}
\newtheorem{thm}{\hspace{6mm}Theorem}[section]

\newtheorem{lem}{\hspace{6mm}Lemma\,}[section]

\theoremstyle{definition}

\theoremstyle{remark}
\newtheorem{exam}{\hspace{6mm}Example}[section]

\title{Conditioning of Finite Volume Element Method for Diffusion Problems with General Simplicial Meshes}

\author{Xiang Wang%
\thanks{School of Mathematics, Jilin University, Changchun 130012, China (wxjldx@jlu.edu.cn).},
\and Weizhang Huang%
\thanks{Department of Mathematics, The University of Kansas, Lawrence, KS 66045 (whuang@ku.edu).},
\and Yonghai Li%
\thanks{School of Mathematics, Jilin University, Changchun 130012, China (yonghai@jlu.edu.cn).}
}

\begin{document}
\maketitle

\begin{abstract}
The conditioning of the linear finite volume element discretization for general diffusion equations is studied on arbitrary simplicial meshes. The 
condition number is defined as the ratio
of the maximal singular value of the stiffness matrix to the minimal eigenvalue of its symmetric part.
This definition is motivated by the fact that the convergence rate of
the generalized minimal residual method for the corresponding linear systems
is determined by the ratio. An upper bound for the ratio
is established by developing an upper bound for the maximal singular value and
a lower bound for the minimal eigenvalue of the symmetric part.
It is shown that the bound depends on three factors, the number of the elements in the mesh,
the mesh nonuniformity measured in the Euclidean metric, and the mesh nonuniformity measured
in the metric specified by the inverse diffusion matrix. It is also shown that the diagonal scaling
can effectively eliminates the effects from the mesh nonuniformity measured in the Euclidean metric.
Numerical results for a selection of examples in one, two, and three dimensions are presented.
\end{abstract}



\noindent{\textbf{ Abbreviated title.}}
Conditioning of FVEM with General Meshes

\section{Introduction}

The finite volume element method (FVEM) is a type of finite volume method that approximates
the solution of partial differential equations (PDEs) in a finite element space. It inherits many advantages
of finite volume methods such as the  local conservation property while avoiding the complexity other types of finite
volume method have in defining the gradient of the approximate solution
needed in the discretization of diffusion equations.
FVEM has been successfully applied to a broad range of problems
and studied extensively in theory; e.g., see \cite{BankRose1987,Barth2004,CJBi2007,Cai1991SIAM,Cai1991,LChen2010, ZYChenMathComp2015,SHChou2007,REEwing2002,Hackbusch1989,LCW2000,Liebau1996,TSchmidt1993,MYangCBiJLiu2009}.
To date, significant progress has been made in understanding FVEM's stability and superconvergence,
establishing error bounds, and developing high-order FVEMs.
For example, stability analysis and error estimates in $L^2$ or $H^1$ norm
are developed for triangular and quadrilateral meshes in \cite{CWX2012,LYZ2015,WL2016,JXu2009, ZhangZou2015}
while superconvergence results are established recently
in \cite{CZZ2015,LL2012,WangLi2017,ZhangZou2015}.

On the other hand, little progress has been made in understanding the conditioning of FVEM
discretization on general meshes. There are two major barriers toward this.
The first one is that FVEM does not preserve the symmetry of the underlying differential operator
and has a nonsymmetric stiffness matrix in general. It is well known that standard condition numbers
provide little useful information for the solution of nonsymmetric algebraic systems.
A common alternative for measuring the conditioning of a nonsymmetric matrix
is the ratio of its largest singular value to the minimal eigenvalue of its symmetric part.
This is largely motivated by the work of Eisenstat et al. \cite{Eisenstat1983} (or see (\ref{Eisenstat-1}) below)
stating that the ratio determines the convergence rate of the generalized minimal residual method (GMRES)
for the corresponding linear systems.
Establishing an upper bound for the ratio requires the development of an upper bound
for the maximal singular value and a lower bound for the minimal eigenvalue of the symmetric
part. This process is more difficult and complicated in general than that used to establish
bounds for the extremal eigenvalues for symmetric and positive definite matrices.

The second barrier comes from mesh nonuniformity. A main advantage of FVEM is its flexibility
to work with (nonuniform) adaptive meshes needed in many applications.
It thus makes sense that the analysis is carried out for general nonuniform meshes without prior requirements
on their uniformity and regularity. However, this is not a trivial task in general since
it will need to have a mathematical characterization for nonuniform meshes and take the interplay
between the mesh geometry and the underlying differential operator
(or the diffusion matrix in the case of diffusion equations) into full consideration.
For example, Li~et~al.~\cite{Liyonghai2012MathComp} study a multilevel preconditioning technique
for FVEM and establish a uniform bound on the ratio of the largest singular value to the minimal
eigenvalue of the symmetric part of the preconditioned stiffness matrix but their analysis
and results are valid only for quasi-uniform meshes.
Moreover, FVEM analysis (such as $L^2$ error estimation) typically obtains relevant
properties from the finite element (FE) discretization of the underlying problem
by estimating the difference between the corresponding bilinear forms.
This type of estimation has so far been carried out only for quasi-uniform or regular meshes too; e.g.
see \cite{LCW2000,Liyonghai2012MathComp,LYZ2015,WL2016}.

It is interesting to point out that much more effort and progress have been made
to understand the conditioning of FE discretization on general meshes.
Noticeably, Fried~\cite{Fri1973} obtains a bound on the condition number of the stiffness matrix for the linear FE approximation of the Laplace operator for a general mesh.
Bank and Scott~\cite{BaSc1989} show that the condition number of the diagonally scaled stiffness matrix
for the Laplace operator on an isotropic adaptive mesh is essentially the same as for a quasi-uniform mesh.
Ainsworth, McLean, and Tran~\cite{AMT2000} and Graham and McLean~\cite{Graham2006}
extend this result to the boundary element equations for locally quasi-uniform meshes.
Du et al.~\cite{Du2009} obtain a bound on the condition number of the stiffness matrix
for a general diffusion operator on a general mesh which reveals the relation between the condition number
and some mesh quality measures. The result is extended by Zhu and Du~\cite{ZD2011,Du-2014} to
parabolic problems. Shewchuk~\cite{Shewchuk2002} provides a bound
on the largest eigenvalue of the stiffness matrix scaled by the lumped mass matrix
in terms of the maximum eigenvalues of local element matrices.
More recently, bounds for the condition number of the stiffness matrix for the linear FE equations of
a general diffusion operator (and implicit Runge-Kutta schemes of the corresponding parabolic problem)
on an arbitrary mesh are developed in \cite{KaHu2013b,HKL2016,KaHuXu2012} while
the largest permissible time steps for explicit Runge-Kutta schemes for both linear and high order FE
approximations of parabolic problems are established in \cite{HKL2013b,HuKaLa2013}.
These bounds take into full consideration of the interplay between
the mesh geometry and the diffusion matrix.
Indeed, they show that the condition number of the stiffness matrix depends on three factors:
the factor depending on the number of mesh elements and corresponding to the condition number
of the linear FE equations for the Laplace operator on a uniform mesh, the mesh nonuniformity
measured in the metric specified by the inverse diffusion matrix, and the mesh nonuniformity measured
in the Euclidean metric. Moreover, the Jacobi preconditioning, or called the diagonal scaling,
can effectively eliminate the effects of mesh nonuniformity and reduce those of the mesh nonuniformity
with respect to the inverse diffusion matrix.

The objective of this paper is to study the conditioning for linear FVEM applied to anisotropic diffusion
problems on general simplex meshes in any dimension. We shall use the ratio of the maximal singular
value to the minimal eigenvalue of the symmetric part of the stiffness matrix to measure
its conditioning (cf. (\ref{cond_0}) below).
The task of estimating the condition number is then to develop an upper bound for the maximal singular value
and a lower bound for the minimal eigenvalue of the symmetric part of the stiffness matrix.
To this end, we use the FE bilinear form and show that the difference between the FE and FVE bilinear forms
is small when the mesh is sufficiently fine. We also use a strategy similar to that in \cite{KaHuXu2012}
for establishing a lower bound for the minimal eigenvalue of the symmetric part of the FVEM stiffness matrix.
The results of this work are similar to those in \cite{KaHuXu2012}. In particular, the bound for the condition
number depends on three factors too, i.e.,  the number of mesh elements and the mesh nonuniformity
measured in the Euclidean metric and in the metric specified by the inverse diffusion matrix.
Moreover, the analysis shows that the diagonal scaling
can effectively eliminate the effects of mesh nonuniformity in the Euclidean metric.
To a large extent, the current work can be viewed as an extension of \cite{KaHuXu2012} from FEM to FVEM.
However, this extension is by no means trivial. As mentioned earlier, we have to deal with the nonsymmetric nature
of the stiffness matrix in the current situation. Moreover, the current analysis is more technical and difficult
since FVEM depends heavily on the specific geometry of the dual mesh elements which are formed by partitioning
primary mesh elements in a certain manner.

The outline of this paper is as follows. The linear FVEM is described in \S\ref{sec:FVEM}
for the boundary value problem of an anisotropic diffusion equation.
The definition of the condition number of the stiffness matrix and its estimates are given in \S\ref{SEC:stiffness}.
A similar analysis is carried out for the mass matrix in \S\ref{SEC:mass}, followed
by a selection of numerical examples in one, two, and three dimensions in \S\ref{sec:numerical}.
Conclusions are made in \S\ref{SEC:conclusions}. Finally, the derivation for the expressions of two parameters
in the mass matrix is given in Appendix A.

\section{Linear finite volume element formulation}
\label{sec:FVEM}

We consider the boundary value problem (BVP) of an anisotropic diffusion equation as
\begin{eqnarray} \label{BVP-1}
  \left\{ \begin{array}{rcl}
  Lu\equiv -\nabla\cdot (\uD \nabla u)&=&f, \quad \text{ in }\; \Omega,\\
  u &=&0, \quad \text{ on }\; \partial \Omega,
  \end{array} \right.
\end{eqnarray}
where $\Omega\subset \mathbb{R}^{d}$ ($d= 1$, 2, 3, \dots) is a bounded polygonal/polyhedral domain,
$f$ is a given function, and $\mathbb{D}=\mathbb{D}(\ux)=(d_{ij})_{d\times d}$ is the diffusion matrix.
We assume that $\mathbb{D}$ is sufficiently smooth, symmetric, and strictly positive definite on $\Omega$
in the sense that there exist positive constants $0 < \underline{d} \le \overline{d}$ such that
\begin{eqnarray} \label{D-1}
\underline{d}\; |\xi|^{2}\leq \xi^{T} \uD (\ux) \xi \leq \overline{d}\; |\xi|^{2},\quad \forall \xi \in \mathbb{R}^{d},
\quad \forall \ux \in \Omega .
\end{eqnarray}

\begin{figure}[thb]
\centering
\begin{subfigure}{0.4\textwidth}
\centering
\includegraphics[scale = 0.15]{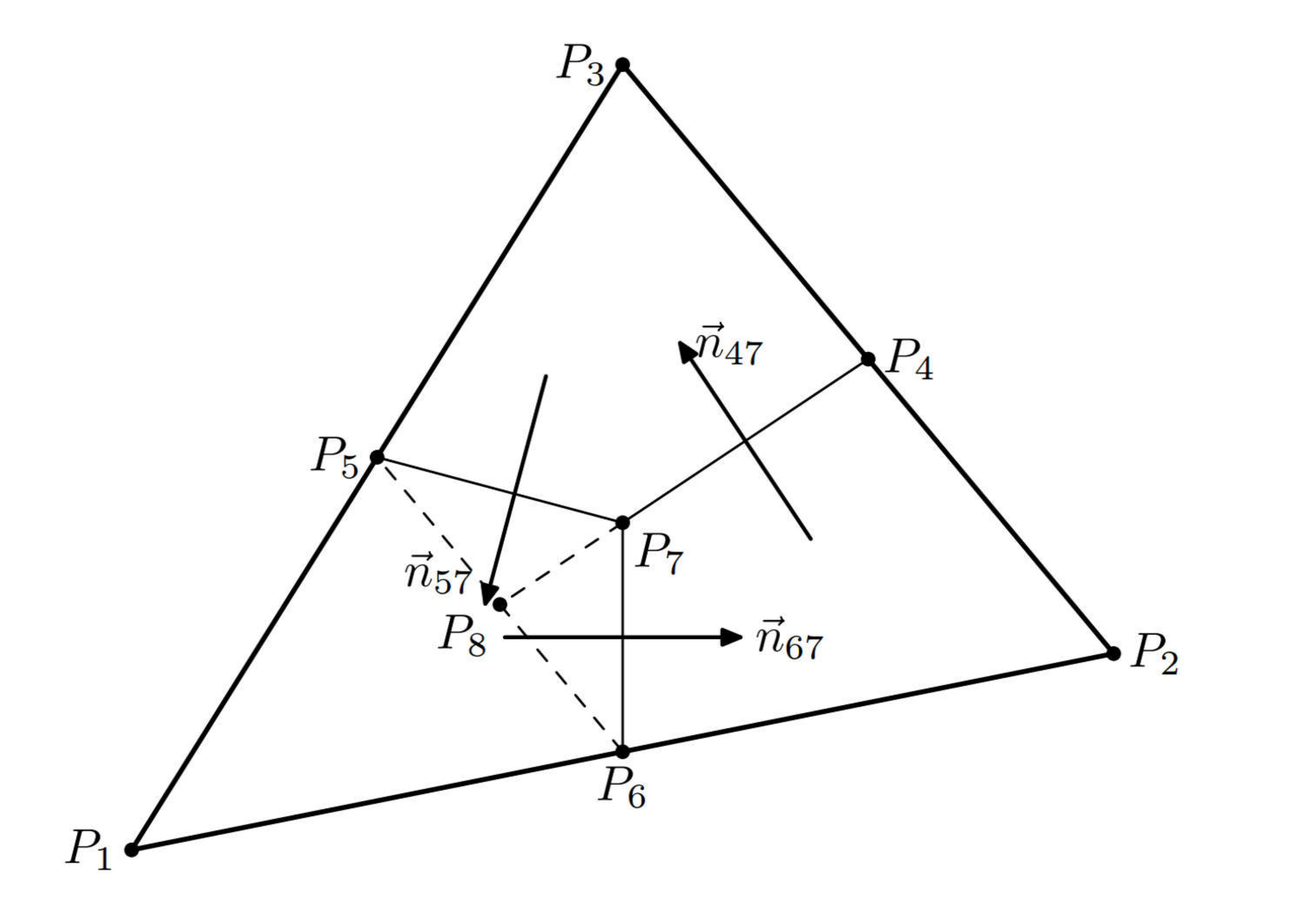}
\caption{}
\end{subfigure}
\begin{subfigure}{0.4\textwidth}
\centering
\includegraphics[scale = 0.2]{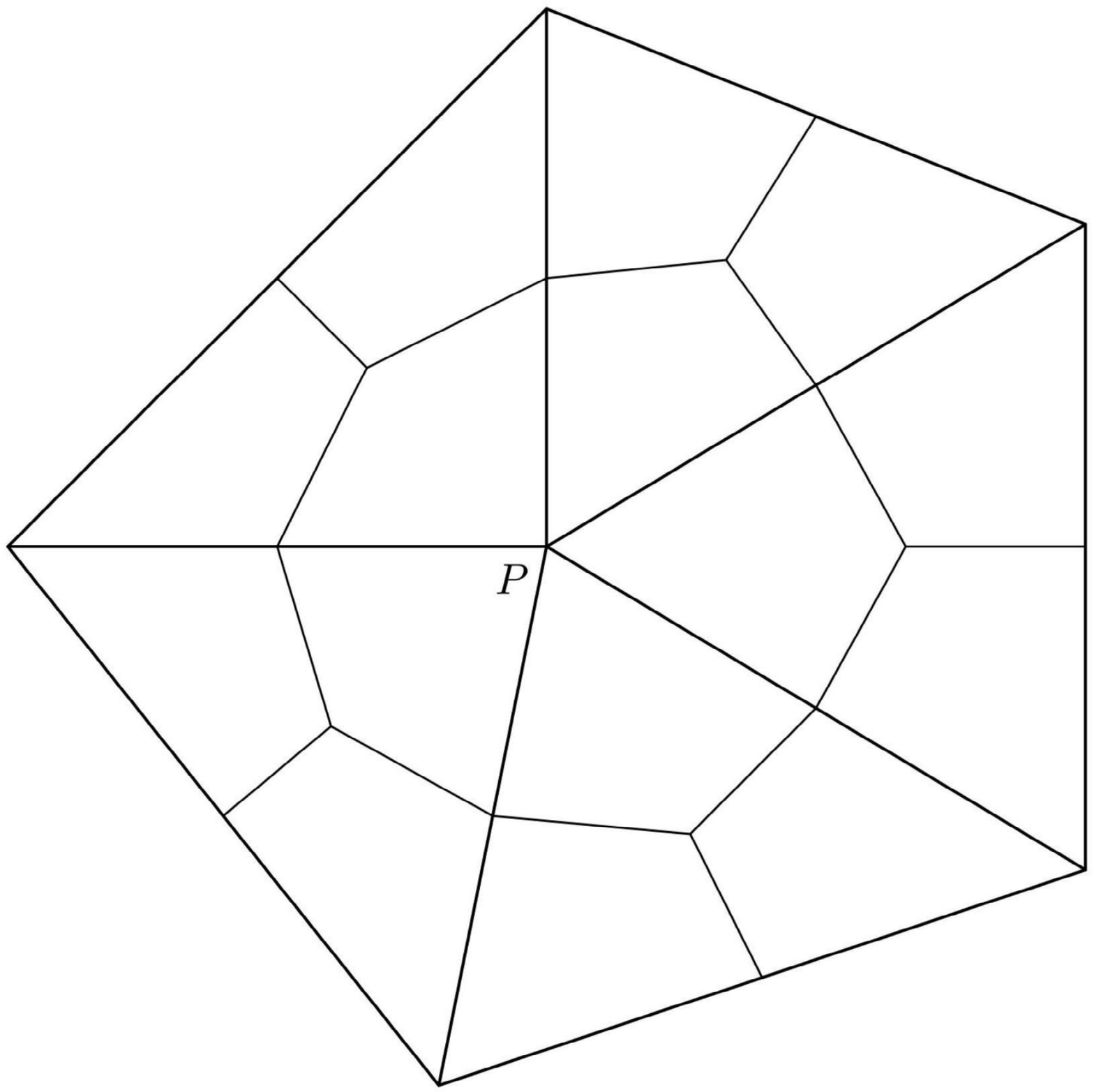}
\caption{}
\end{subfigure}
\caption{Illustration of dual element of Linear FVEM on element $K$ or at vertex $P$ (2-dimension).}
\label{fig:LinearFVMDual}
\end{figure}

Let $\mathcal{T}_{h}=\{K\}$ be a given simplicial mesh for $\Omega$,
$\mathcal{N}_{K}$ be the set of all computing nodes in $K$, and
$\mathcal{N}=\cup_{K\in\mathcal{T}_{h}}\mathcal{N}_{K}$.
For linear FVEM discretization, $\mathcal{N}_{K}$ is the set of the $(d+1)$ vertices of $K$.
Divide each simplex into $(d+1)$ sub-regions by plane or line segments connecting the centroids of the simplex
and its faces and edges. The dual element $K_P^{*}$ associated with the vertex $P$ is
formed by the sub-regions surrounding $P$. The dual mesh is then defined as
$\mathcal{T}_{h}^{*}=\{K_P^{*}\; | \; \forall P \in \mathcal{N} \}$ for linear FVEM.
The structures of the dual mesh on element $K \in \mathcal{T}_{h}$ (left) and at vertex $P$ (right) in two dimensions are illustrated in Fig.~\ref{fig:LinearFVMDual}, where $P_{i}$ $(i=4,\, 5,\, 6)$ are the edge midpoints
and $P_{7}$ is the centroid (i.e., the barycenter) of $K$. The trial and test function spaces are chosen as
\begin{eqnarray}
\mathit{U}^{h} &=& \left \{ u^{h} \; | \; u^{h}\in C(\Omega), \, u^{h}\vert_{K} \in P^{1}(K),\, \forall K\in \mathcal{T}_{h}
\quad\!\!\! \mathrm{and}\quad\!\!\! u^{h}(P) = 0, \forall P\in \mathcal{N}\cap(\partial \Omega)\right \}, \nonumber\\
\mathit{V}^{h} &=& \left \{ v^{h} \; | \;  v^h \in L^{2}(\Omega) ,\, v^h\vert_{\mathit{K}_{P}^{*}}=\mathrm{constant},
\forall K_{P}^{*}\in\mathcal{T}_{h}^{*}\quad\!\!\! \mathrm{and}\quad\!\!\!  v^h\vert_{\mathit{K}_{P}^{*}}=0,
\forall P\in\partial \Omega  \right \}.  \nonumber
\end{eqnarray}
Denote the mapping from $U^{h}$ to $V^{h}$ by $\Pi_{h}^{*}$, i.e.,
\[
\Pi_{h}^{*} u^h = u^h(P), \quad \forall \ux \in K_P^{*}, \quad P\in \mathcal{N} .
\]
We also denote the diameter of element $K$ by $h_K$ and define
\[
|\uD|_{1,\infty,K}=\max\limits_{i,j=1,..., d} |d_{ij}|_{1,\infty,K}, \quad
|\uD|_{2,\infty,K}=\max\limits_{i,j=1,..., d} |d_{ij}|_{2,\infty,K},\quad
\uD_{K}=\frac{1}{|K|}\int_{K}\uD \,\ud \ux ,
\]
where $|\cdot |_{m,p,K}$ is the semi-norm of Sobolev space $W^{m,p}(K)$, $|K|$ is the volume
(or the $d$-dimensional measure) of $K$, and $\uD_{K}$ is the average of $\uD$ over $K$.

The linear FVEM approximation of (\ref{BVP-1}) is to find $u^{h} \in U^{h}$ such that
\begin{eqnarray}
\label{FVEM-0}
a_{h}(u^{h},v^{h})=(f,v^{h}), \qquad  \forall v^{h} \in V^{h},
\end{eqnarray}
where
\begin{equation}
a_{h}(u^{h},v^{h}) = - \sum\limits_{K_{P}^{*}\in \mathcal{T}_{h}^{*}}
\int_{\partial K_{P}^{*}} (\uD \nabla u^{h})\cdot {\mathbf{n}} \, v^{h} \ud s ,
\quad
(f,v^{h}) =  \sum\limits_{K_{P}^{*}\in \mathcal{T}_{h}^{*}} \int_{K_{P}^{*}} f\, v^{h}
\ud \ux ,
\label{FVEM-1}
\end{equation}
and ${\mathbf{n}}$ is the unit outward normal of $\partial K_{P}^{*}$.
Denote the number of the elements and interior vertices (computing nodes) of $\mathcal{T}_{h}$ by $N$ and
$N_{vi}$, respectively. Assume that the vertices are ordered in such a way that the first $N_{vi}$
vertices are the interior vertices. Then, $U^{h}$ and $u^{h}$ can be expressed as
\begin{eqnarray}
U^{h} &=& \text{span}\{\phi_{1},\dots,\phi_{N_{vi}}\},  \\
u^{h}     &=& \sum_{j=1}^{N_{vi}} u_{j}\phi_{j},
\label{uhexp}
\end{eqnarray}
where $\phi_{j}$ is the linear basis function associated with the $j$-th vertex $P_j$.
Substituting $(\ref{uhexp})$ into $(\ref{FVEM-0})$
and taking $v^h$ as the characteristic function of $K_{P_i}^{*}$ ($i = 1, ..., N_{vi}$) successively,
we can rewrite (\ref{FVEM-0}) into matrix form as
\begin{eqnarray}
A_{FV} \; \boldsymbol{u}=\boldsymbol{f},  \label{elliptic_matrix}
\end{eqnarray}
where $\boldsymbol{u} = ( u_{1},\dots,u_{N_{vi}} )^{T}$, $\boldsymbol{f} = ( f_{1},\dots,f_{N_{vi}} )^{T}$,
and the entries of the stiffness matrix $A$ and the right-hand-side vector $\boldsymbol{f}$ are given by
\begin{eqnarray}
a_{ij}^{FV} &=& - \int_{\partial K_{P_{i}}^{*}} (\uD \nabla \phi_{j})\cdot {\mathbf{n}}\, \ud s,
            \quad i,j=1,\dots,N_{vi},
\label{A-1}
\\
f_{i}  &=&  \int_{K_{P_{i}}^{*}} f\, \ud \ux, \quad i=1,\dots,N_{vi} .
\end{eqnarray}
Let $\omega_i$ be the element patch associated with $P_i$ and $\omega_{ij} = \omega_i \cap \omega_j$.
Then we can rewrite $a_{ij}^{FV}$ as
\begin{eqnarray}
a_{ij}^{FV} = \sum\limits_{K \in \omega_{ij}} a_{ij, K}^{FV}, \qquad
a_{ij, K}^{FV} &=& -\int_{\partial K_{P_{i}}^{*}\cap K}(\uD \nabla \phi_{j})\cdot {\mathbf{n}}\, \ud s .
\label{eq:aijK}
\end{eqnarray}

\begin{lem}
The stiffness matrix $A_{FV}$ is symmetric when $\uD$ is piecewise constant on $\mathcal{T}_{h}$.
\end{lem}

\begin{proof}
Denote the face of $K$ opposite to $P_{i}$ by $l_{i,K}$ and the distance from $P_{i}$ to $l_{i,K}$
by $\d(P_{i},l_{i,K})$. It is easy to see that
\[
|K|=\frac{1}{d}\,\d(P_{i},l_{i,K}) \, | l_{i,K} |,\quad
\nabla \phi_i = - \frac{\mathbf{n}_{l_{i,K}}}{\d(P_{i},l_{i,K})} ,
\]
where $ |l_{i,K} |$ is the area (for 3D) or
the $(d-1)$-dimensional measure of $l_{i,K}$, and $\mathbf{n}_{l_{i,K}}$ is the unit outward normal
of $l_{i,K}$. Let $S_{i,k}$ be the $k$-th face of $(\partial K_{P_i}^{*})\cap K$ ($k = 1, ..., d$).
It is noted that $S_{i,k}$'s are different from  $l_{i,K}$'s: $S_{i,k}$ is in the interior of $K$ while
$l_{i,K}$ is a part of $\partial K$. Moreover, $\cup_{k=1}^d S_{i,k}$ separates the dual element
corresponding to $P_{i}$ from other dual elements restricted in $K$. We have
\begin{eqnarray}
\int_{\partial K_{P_i}^{*} \cap K }  {\mathbf{n}} \, \ud s
    &=&  \sum_{k=1}^d ( {\mathbf{n}}_{S_{i,k}}  \,\int_{S_{i,k}} 1\, \ud s)
     =   \sum_{k=1}^d  {\mathbf{n}}_{S_{i,k}}  \,|S_{i,k}| \nonumber \\
     &=&   \frac{2}{d(d+1)} \sum_{k=1}^d  {\mathbf{n}}_{S_{i,k}}  \, \frac{d(d+1)}{2}|S_{i,k}|          \nonumber\\
    &=&  \frac{2}{d(d+1)} \, \sum_{k=1}^d ( {\mathbf{n}}_{l_{i,K}}\, |l_{i,K}|+
         \frac{1}{2}\sum_{\mbox{\tiny$\begin{array}{c}
            t=1,\dots,d\\
            t\not=k \end{array}$}}
            {\mathbf{n}}_{l_{t,K}}  \,|l_{t,K}|)          \nonumber\\
    &=&  \frac{2}{d(d+1)} \, ( d {\mathbf{n}}_{l_{i,K}}\, |l_{i,K}| +
            \frac{d-1}{2}\sum_{k=1}^d {\mathbf{n}}_{l_{k,K}}  \,|l_{k,K}|)    \nonumber\\
    &=&  \frac{2}{d(d+1)} \, ( d {\mathbf{n}}_{l_{i,K}}\, |l_{i,K}| -
            \frac{d-1}{2}{\mathbf{n}}_{l_{i,K}}\, |l_{i,K}|)          \nonumber\\
    &=&  \frac{1}{d} \quad    {\mathbf{n}}_{l_{i,K}}\, |l_{i,K}|
    = \frac{1}{d} \, (-\d(P_{i},l_{i,K})\, \nabla \phi_{i} ) \, |l_{i,K}| = -|K| \,  \nabla \phi_{i},    \nonumber
\end{eqnarray}
where we have used the equalities
\begin{eqnarray}
& {\mathbf{n}}_{S_{i,k}}  \, \frac{d(d+1)}{2}|S_{i,k}| = {\mathbf{n}}_{l_{i,K}}\, |l_{i,K}|+
         \frac{1}{2}\sum\limits_{\mbox{\tiny$\begin{array}{c}
            t=1,\dots,d,  \\
             t\not=k \end{array}$}}
            {\mathbf{n}}_{l_{t,K}}  \,|l_{t,K}|,         & \nonumber\\
& {\mathbf{n}}_{l_{i,K}}\, |l_{i,K}| + \sum_{k=1}^d {\mathbf{n}}_{l_{k,K}}  \,|l_{k,K}| = 0. & \nonumber
\end{eqnarray}
(The second equality states the fact that the sum of the unit outward normal vectors of all faces
multiplied by their $(d-1)$-dimensional measures vanishes for any polyhedron.)

When $\uD$ is piecewise constant on $\mathcal{T}_{h}$, we get, for $i \neq j$,
\begin{eqnarray}
a^{FV}_{ij,K} &=& - \int_{\partial K_{P_i}^{*} \cap K } (\uD\nabla \phi_{j})\cdot {\mathbf{n}} \, \ud s
              = - (\uD\nabla \phi_{j}) |_K\cdot \int_{\partial K_{P_i}^{*} \cap K }  {\mathbf{n}} \, \ud s , \nonumber\\
              &=&  |K| (\uD\nabla \phi_{j})|_K \cdot \nabla \phi_{i}|_K
              =  |K| (\uD\nabla \phi_{i})|_K \cdot \nabla \phi_{j}|_K = a^{FV}_{ji,K}. \nonumber
\end{eqnarray}
From this and (\ref{eq:aijK}), we get $a^{FV}_{ij}= a^{FV}_{ji}$, which implies that $A_{FV}$ is symmetric.
\end{proof}

The above proof also shows that $A_{FV}$ is not symmetric in general when
$\uD$ is not piecewise constant.

To conclude this section, we prove two orthogonality properties which are needed in the later analysis.

\begin{lem}
For any $K \in \mathcal{T}_{h}$, there hold
\begin{eqnarray}
\label{eq:orth_internal}
&\int_{\mathit{K}}g (v -\Pi_{h}^{*}v ) \ud \ux=0,\quad  &\forall g \in P^{0}(\mathit{K}),\;  \forall v \in P^{1}(\mathit{K}) , \\
\label{eq:orth_bndry}
&\int_{l_{i,K}}g (v -\Pi_{h}^{*}v ) \ud s =0,\quad &\forall g\in P^{0}(l_{i,K}), \; \forall v \in P^{1}(l_{i,K}),\;
\forall l_{i,K}\subset\partial K .
\end{eqnarray}
Here, $P^0(K)$ and $P^1(K)$ are the constant space and linear space on $K$.
\end{lem}
\begin{proof}
Denote the vertices of $K$ by $P_{i}, \; i=1,\dots,d+1$.
For any $v \in P^{1}(\mathit{K})$, we have
\[
\int_{\mathit{K}}v  \ud \ux = \sum_{i=1}^{d+1} \left( \frac{|K|}{d+1}\, v(P_{i})\right)
       = \sum_{i=1}^{d+1}  v(P_{i}) \int_{\mathit{K}\cap K_{P_{i}}^{*}} 1 \ud \ux
       = \int_{\mathit{K}} \Pi_{h}^{*}v\; \ud \ux ,
\]
which implies
\[
\int_{\mathit{K}} \, (v-\Pi_{h}^{*}v) \, \ud \ux=0.
\]
Since $g\in P^{0}(\mathit{K})$ is constant on $K$, the above equality gives (\ref{eq:orth_internal}).

Similarly, we can obtain (\ref{eq:orth_bndry}).
\end{proof}

We consider the piecewise linear and piecewise constant approximations of $\uD$ on $\mathcal{T}_{h}$ as
\begin{align*}
& \uD_1(\ux) = \sum_{j=1}^{N_v} \uD (P_j) \phi_j (\ux) ,
\\
& \uD_0(\ux) = \frac{1}{|K|} \int_K \uD( \tilde{\ux}) d \tilde{\ux},\quad  \forall \ux \in K, \; K \in \mathcal{T}_{h} .
\end{align*}
Then from (\ref{eq:orth_internal}) and (\ref{eq:orth_bndry}) we have
\begin{align}
\sum\limits_{K\in\mathcal{T}_{h}} \int_{K}
                \nabla\cdot(\uD_1 \nabla v^{h})(v^{h}-\Pi_{h}^{*}v^{h}) \ud \ux = 0,\quad \forall v^h \in U^h,
\label{orth_bndry_1}
\\
\sum\limits_{K\in\mathcal{T}_{h}} \int_{\partial K}  (\uD_0 \nabla v^{h})\cdot{\mathbf{n}}
\, (v^{h}-\Pi_{h}^{*}v^{h}) \ud s =0,   \quad \forall v^h \in U^h,
\label{orth_internal_1}
\end{align}
where $\uD_0 |_{\partial K}$ is understood as $\uD_0 |_{\partial K} = \uD_0 |_{K}$.

\section{Conditioning of the stiffness matrix}
\label{SEC:stiffness}

In this section we study the conditioning of the stiffness matrix $A_{FV}$ of linear FVEM.
As shown in the previous section, $A_{FV}$ is generally nonsymmetric for a non-piecewise-constant diffusion matrix.
It is well known that a condition number in the standard definition
does not provide much information for the convergence of iterative methods for nonsymmetric systems.
On the other hand, when its symmetric part, $(A_{FV} + A_{FV}^T)/2$, is positive definite, which is to be shown
later in this section, the convergence of the generalized minimal residual method (GMRES)
is given by Eisenstat et al. \cite{Eisenstat1983} as
\begin{equation}
\| r_n \|_2 \le \left (1 - \frac{\lambda_{min}^2((A_{FV} + A_{FV}^T)/2)}{\sigma_{max}^2(A_{FV})}
\right )^{n/2} \| r _0\|_2 ,
\label{Eisenstat-1}
\end{equation}
where $\sigma_{max}(A_{FV})$ is the largest singular value of $A_{FV}$, $\lambda_{min}((A_{FV} + A_{FV}^T)/2)$
is the minimal eigenvalue of the symmetric part, $r_n$ is the residual
of the corresponding linear system at the $n$-th iterate, and $\| \cdot \|_2$ stands for the matrix or vector 2-norm.
From this, we can consider the ``condition number''
\begin{equation}
{\kappa}(A_{FV}) = \frac{\sigma_{max}(A_{FV})}{\lambda_{min}((A_{FV} + A_{FV}^T)/2)} .
\label{cond_0}
\end{equation}
This definition reduces to the standard definition of the condition number (in 2-norm)
for symmetric matrices. For notational simplicity
and without causing confusion, we use the standard notation for this definition here and will hereafter simply refer
this as the condition number of $A_{FV}$.

In the following we shall show that the symmetric part of $A_{FV}$ is positive definite when the mesh is sufficiently fine.
We shall also establish an upper bound for $\sigma_{max}(A_{FV})$ and a lower bound
for $\lambda_{min}((A_{FV} + A_{FV}^T)/2)$. Similar bounds will be obtained for the situation with
the Jacobin (diagonal) preconditioning. As an additional benefit, the bounds will be used to reveal
the effects of the interplay between the mesh geometry and the diffusion matrix on
the conditioning of $A_{FV}$.

In our analysis, we use results for the conditioning of the stiffness matrix ($A_{FE}$)
of a linear finite element approximation of (\ref{BVP-1}). This topic has been studied by
a number of researchers; e.g., see
\cite{Ainsworth1999,BaSc1989,Du2009,Fri1973,KaHuXu2012,HuKaLa2013,Wathen87,Du-2014}.
Recall that the entries of $A_{FE}$ are given by
\begin{equation}
a_{ij}^{FE} = \int_\Omega (\uD \nabla \phi_j) \cdot \nabla \phi_i \ud \ux
= \sum\limits_{K \in \omega_{i j}} \int_K (\uD \nabla \phi_j) \cdot \nabla \phi_i \ud \ux ,
\quad i, j = 1, ..., N_{vi}
\label{FE-1}
\end{equation}
and $A_{FE}$ is symmetric and positive definite for any diffusion matrix.

Denote the set of the indices of the neighboring vertices of $P_j$ (excluding $P_j$)
by $\mathcal{N}_{j}^{0}$ and define $\mathcal{N}_{j}=\{j\}\cup\mathcal{N}_{j}^{0}$.
Let $p_{\mathcal{N}_{j}}$ be the number of the elements (indices of points) in $\mathcal{N}_{j}$
and $p_{\max}=\max\limits_{j=1,\dots,N_{vi}}p_{\mathcal{N}_{j}}$. Let
\begin{equation}
\begin{cases}
& C_{0}= \frac{\sqrt{p_{\max}}}{\underline{d}}, \qquad
C_{\tilde{\nabla}} = \frac{d}{d+1} \left (\frac{\sqrt{d+1}}{d!}\right )^{\frac 2 d},
 \\
& C_{\uD,h_{K}}=d^2 \, ( h_{K}\,|\uD|_{2,\infty,K} +  |\uD|_{1,\infty,K}),\\
& H_{h}=\max\limits_{K\in\mathcal{T}_{h}} C_{\uD,h_{K}}\,h_{K} .
\end{cases}
\label{CDh-1}
\end{equation}
Notice that $H_h \to 0$ as $h \equiv \max_K h_K \to 0$.

\begin{lem}
\label{lem:diff-1}
There holds
\begin{equation}
|a^{FV}_{ij}-a^{FE}_{ij}| \le
\sum\limits_{K\in\omega_{ij}} C_{\uD,h_{K}} \,|K|^{1/2}\,|\phi_{j}|_{1,K} , \quad \forall i,j = 1, ..., N_{vi}.
\label{lem:diff-1-1}
\end{equation}
\end{lem}

\begin{proof}
Using the definitions of $a^{FV}_{ij}$ and $a^{FE}_{ij}$ and the divergence theorem, we have
\begin{align}
a^{FV}_{ij} & =  {-\int_{\partial K_{P_{i}}^{*}}(\uD\nabla \phi_{j})\cdot {\mathbf{n}} \ud s}
=   - \sum_{K^{*}\in \mathcal{T}_{h}^{*}}\int_{\partial K^{*}}(\uD\nabla \phi_{j})\cdot {\mathbf{n}} \,\Pi_{h}^{*}\phi_{i}  \ud s
\nonumber \\
            & = - \sum\limits_{K\in\omega_{ij}} \sum_{K^{*}\in \mathcal{T}_{h}^{*}}
                    \int_{\partial K^{*}\cap K}(\uD\nabla \phi_{j})\cdot {\mathbf{n}} \,\Pi_{h}^{*}\phi_{i}  \ud s \nonumber\\
            & =  \sum\limits_{K\in\omega_{ij}} \int_{\partial K} (\uD \nabla \phi_{j})\cdot{\mathbf{n}}
                 \, \Pi_{h}^{*}\phi_{i} \ud s
                 - \sum\limits_{K\in\omega_{ij}} \int_{K}
                \nabla\cdot(\uD \nabla \phi_{j})\, \Pi_{h}^{*}\phi_{i} \ud \ux,
\nonumber
\\
a^{FE}_{ij} & =  \sum_{K\in \omega_{ij}}\int_{K}(\uD\nabla \phi_{j})\cdot \nabla \phi_{i} \ud \ux \nonumber\\
            & =  \sum\limits_{K\in\omega_{ij}} \int_{\partial K} (\uD \nabla \phi_{j})\cdot{\mathbf{n}} \, \phi_{i} \ud s
                 -\sum\limits_{K\in\omega_{ij}} \int_{K}  \nabla\cdot(\uD \nabla \phi_{j})\, \phi_{i} \ud \ux.   \nonumber
\end{align}
Noticing that $|\phi_{i}-\Pi_{h}^{*}\phi_{i}|\leq 1$, we get
\[
|\phi_{i}-\Pi_{h}^{*}\phi_{i}|_{0,K}=\left (\int_{K} \, (\phi_{i}-\Pi_{h}^{*}\phi_{i})^{2} \, \ud \ux\right )^{1/2}\leq
\left (\int_{K} \, 1 \, \ud \ux\right )^{1/2}  \leq  |K|^{1/2}.
\]
Using the orthogonality properties (\ref{orth_bndry_1}) and (\ref{orth_internal_1}),
the trace theorem, and the fact that $\nabla\phi_{j}$ is constant in each element, we have
\begin{align*}
|a^{FV}_{ij}-a^{FE}_{ij}| &= |\sum\limits_{K\in\omega_{ij}} \int_{\partial K}
                    (\uD \nabla \phi_{j})\cdot{\mathbf{n}}
                 \, (\phi_{i}-\Pi_{h}^{*}\phi_{i}) \ud s
                 -\sum\limits_{K\in\omega_{ij}} \int_{K}
                \nabla\cdot(\uD \nabla \phi_{j})\,(\phi_{i}-\Pi_{h}^{*}\phi_{i}) \ud \ux|
\\
    &\leq |\sum\limits_{K\in\omega_{ij}} \int_{K}
                \nabla\cdot((\uD-\uD_{1}) \nabla \phi_{j})(\phi_{i}-\Pi_{h}^{*}\phi_{i}) \ud \ux|    \nonumber\\
    & +|\sum\limits_{K\in\omega_{ij}} \int_{\partial K}  ((\uD-\uD_{0}) \nabla \phi_{j})\cdot{\mathbf{n}}
    \, (\phi_{i}-\Pi_{h}^{*}\phi_{i}) \ud s |     \nonumber\\
    &\leq \sum\limits_{K\in\omega_{ij}} (d^2\, h_{K}\,|\uD|_{2,\infty,K}\,|\phi_{j}|_{1,K}|\phi_{i}-\Pi_{h}^{*}\phi_{i}|_{0,K}
        + d^2\, |\uD|_{1,\infty,K}\,|\phi_{j}|_{1,K}|\phi_{i}-\Pi_{h}^{*}\phi_{i}|_{0,K})        \nonumber\\
    &\leq \sum\limits_{K\in\omega_{ij}} C_{\uD,h_{K}} \,|K|^{1/2}\,|\phi_{j}|_{1,K},
\end{align*}
which gives (\ref{lem:diff-1-1}).
\end{proof}

\begin{lem}
\label{lem:diff-2}
Let $a(\cdot,\cdot)$ be the bilinear form of FEM associated with the BVP $(\ref{BVP-1})$ and
$a_h(\cdot,\cdot)$ be the bilinear form of FVEM defined in $(\ref{FVEM-1})$. Then,
\begin{equation}
|a_{h} (u^{h},\Pi_{h}^{*}u^{h})-a(u^{h},u^{h})| \le H_{h}\,|u^{h}|_{1,\Omega}^{2},\quad \forall u^h \in U^h .
\label{lem:diff-2-1}
\end{equation}

\end{lem}

\begin{proof}
The proof is similar to that of Lemma~\ref{lem:diff-1}. Indeed, for any $u^h \in U^h$,
from the trace theorem and the orthogonality properties (\ref{orth_bndry_1}) and (\ref{orth_internal_1}) we have
\begin{eqnarray*}
 |a_{h} (u^{h},\Pi_{h}^{*}u^{h})-a(u^{h},u^{h})|
    &=& |\sum\limits_{K\in\mathcal{T}_{h}} \int_{K}
        \nabla\cdot(\uD \nabla u^{h})(u^{h}-\Pi_{h}^{*}u^{h}) \ud \ux \\
  &&\qquad  - \, \sum\limits_{K\in\mathcal{T}_{h}} \int_{\partial K}  (\uD \nabla u^{h})\cdot{\mathbf{n}} \, (u^{h}-\Pi_{h}^{*}u^{h}) \ud s|    \nonumber\\
    &\leq& |\sum\limits_{K\in\mathcal{T}_{h}} \int_{K}
                \nabla\cdot((\uD-\uD_{1}) \nabla u^{h})(u^{h}-\Pi_{h}^{*}u^{h}) \ud \ux|    \nonumber\\
    && \qquad +\, |\sum\limits_{K\in\mathcal{T}_{h}} \int_{\partial K}  ((\uD-\uD_{0}) \nabla u^{h})\cdot{\mathbf{n}} \, (u^{h}-\Pi_{h}^{*}u^{h}) \ud s |   \nonumber\\
    &\leq& \sum\limits_{K\in\mathcal{T}_{h}} (d^2 h_{K}^{2}|\uD|_{2,\infty,K}\,|u^{h}|_{1,K}^{2}
            + d^2 h_{K}|\uD|_{1,\infty,K}\,|u^{h}|_{1,K}^{2})  \nonumber\\
    &\leq& \sum\limits_{K\in\mathcal{T}_{h}} \left(C_{\uD,h_{K}} h_{K} |u^{h}|_{1,K}^{2} \right)
    \\
    & \leq &  H_{h}\,|u^{h}|_{1,\Omega}^{2}.
\nonumber
\end{eqnarray*}
\end{proof}

These two lemmas indicate that $A_{FV}$ and $A_{FE}$ are ``close'' when the mesh is sufficiently fine.
Thus, we can establish properties of $A_{FV}$ via estimating the difference between $A_{FV}$ and $A_{FE}$.

\subsection{Largest singular value of the stiffness matrix}

We assume that the reference element $\hat{K}$ has been chosen to be equilateral and unitary.
Denote the affine mapping between $\hat{K}$
and element $K$ by $F_K$ and its Jacobian matrix by $F_K^{'}$.

\begin{thm}
\label{thm:eigFV_max}
Assume that the mesh is sufficiently fine so that $H_{h} < \underline{d}$, where $\underline{d}$ is
the minimum eigenvalue of $\uD$ (cf. $(\ref{D-1})$).
Then, the largest singular value of the stiffness matrix $A_{FV} = (a^{FV}_{ij})$ for the linear FVEM
approximation of BVP $(\ref{BVP-1})$ is bounded above by
\begin{align}
\label{eigFV_max}
\sigma_{\max}(A_{FV}) \; \leq\; &
C_{\tilde{\nabla}} (d + 1) (1+C_{0}H_{h})  \max\limits_{j}
\sum_{K \in \omega_j} |K|\,\|(F_{K}^{'})^{-1} \uD_{K} (F_{K}^{'})^{-T}\|_{2} ,
\\
\label{eigSFVS_max}
\sigma_{\max}(S^{-1}A_{FV}S^{-1}) \; \leq\; & \frac{(1+ C_{0}H_{h})}{( 1 - \underline{d}^{-1}H_{h})} (d + 1),
\end{align}
where $S$ is the Jacobi preconditioner for $A_{FV}$, i.e., $S = (A_{FV}^D)^{1/2}$, with $A_{FV}^D$ being
the diagonal part of $A_{FV}$.
\end{thm}

\begin{proof}

Let $A_{\delta} = A_{FV}-A_{FE}$.  Then, from H\"{o}lder's inequality and Lemma~\ref{lem:diff-1} we have
\begin{align}
\|A_{\delta}\uv\|^2 &=\sum_{i=1}^{N_{vi}}
                        \left (\sum_{j\in\mathcal{N}_{i}} v_{j}(a^{FV}_{ij}-a^{FE}_{ij}) \right )^2 \nonumber\\
                    &\leq \sum_{i=1}^{N_{vi}}  p_{_{\mathcal{N}_{i}}}
                        \sum_{j\in\mathcal{N}_{i}} v_{j}^2(a^{FV}_{ij}-a^{FE}_{ij})^2  \nonumber\\
                    &\leq p_{\max}  \sum_{j=1}^{N_{vi}}  v_{j}^2
                        \sum_{i\in\mathcal{N}_{j}} (a^{FV}_{ij}-a^{FE}_{ij})^2  \nonumber\\
                    &\leq p_{\max}  \sum_{j=1}^{N_{vi}}  v_{j}^2
                        \sum_{i\in\mathcal{N}_{j}} \left (\sum\limits_{K\in\omega_{ij}} C_{\uD,h_{K}}
                        \,|K|^{1/2}\,|\phi_{j}|_{1,K} \right )^2,
\label{eq:aijFVaijFE}
\end{align}
where $p_{_{\mathcal{N}_{i}}}$ is the number of the elements (indices of points) in $\mathcal{N}_i$  and $p_{\max}$ is the maximum value of all $p_{_{\mathcal{N}_{i}}}$ (defined upon (\ref{CDh-1})).

Next we establish a lower bound for $a_{jj}^{FE}$. We have
\begin{equation}
a^{FE}_{jj} = \sum\limits_{K\in \omega_{j}} \int_{K}(\uD \nabla \phi_{j}) \cdot \nabla \phi_{j}  \ud \ux
\geq \underline{d}  \sum\limits_{K\in \omega_{j}} \int_{K} (\nabla \phi_{j}) \cdot \nabla \phi_{j}  \ud \ux
= \underline{d} \sum\limits_{K\in \omega_{j}} |\phi_j|_{1,K}^2.
\label{eq:AFEv-1}
\end{equation}
We observe that when going through all elements in $\mathcal{N}_{j}$, each mesh element in $\omega_{j}$
will be encountered $(d+1)$ times (due to the fact that each element has $(d+1)$ vertices).
Then, from Jensen's inequality we have
\begin{align}
(a^{FE}_{jj})^2  & \ge \underline{d}^2 \left ( \sum\limits_{K\in \omega_{j}} |\phi_j|_{1,K}^2 \right )^2
    = \underline{d}^2 \left (\frac{1}{d+1} \sum\limits_{i\in\mathcal{N}_{j}}
        \sum\limits_{K\in\omega_{ij}} |\phi_{j}|_{1,K}^{2}  \right )^2
 \nonumber \\
    & \geq \frac{\underline{d}^2}{(d+1)^{2}} \sum\limits_{i\in\mathcal{N}_{j}}
        \left (\sum\limits_{K\in\omega_{ij}} |\phi_{j}|_{1,K}^{2}  \right )^2.
\label{eq:AFEv}
\end{align}
Moreover,
\begin{equation}
|\phi_{j}|_{1,K}=\left (\int_{K}|\nabla \phi_{j}|^{2} \ud \ux\right )^{1/2}
\geq \left (\int_{K}(\frac{1}{h_{K}})^{2} \ud \ux\right )^{1/2}=\frac{1}{h_{K}}\,|K|^{1/2}.
\label{eq:AFEv-2}
\end{equation}
Denoting the diagonal part of $A_{FE}$ by $A_{FE}^{D}$ and combining (\ref{eq:aijFVaijFE}),
(\ref{eq:AFEv}), and (\ref{eq:AFEv-2}), we have
\begin{eqnarray}
\label{eq:Adelta_AfeD}
\sup_{\uv\not=0} \frac{\|A_{\delta}\uv\|^2}{\|A_{FE}^D\uv\|^2}
        &\leq& \sup_{\uv\not=0} \frac{p_{\max}  \sum\limits_{j=1}^{N_{vi}} v_{j}^2 \sum\limits_{i\in\mathcal{N}_{j}}
        \left (\sum\limits_{K\in\omega_{ij}} C_{\uD,h_{K}} \,|K|^{1/2}\,|\phi_{j}|_{1,K}\right )^2 }
            {\sum\limits_{j=1}^{N_{vi}} v_{j}^2 (a^{FE}_{jj})^2}  \nonumber\\
        &\leq& p_{\max} \max\limits_{j=1,\dots,N_{vi}}
            \frac{ \sum\limits_{i\in\mathcal{N}_{j}}
        \left (\sum\limits_{K\in\omega_{ij}} C_{\uD,h_{K}} \,|K|^{1/2}\,|\phi_{j}|_{1,K}\right )^2}
            { (a^{FE}_{jj})^2}   \nonumber\\
        &\leq& \frac{(d+1)^2 p_{\max}}{\underline{d}^2} \max\limits_{j=1,\dots,N_{vi}}
            \frac{ \sum\limits_{i\in\mathcal{N}_{j}} \left (\sum\limits_{K\in\omega_{ij}}
            C_{\uD,h_{K}} \,|K|^{1/2}\,|\phi_{j}|_{1,K}\right )^2}
             { \sum\limits_{i\in\mathcal{N}_{j}} \left (\sum\limits_{K\in\omega_{ij}} |\phi_{j}|_{1,K}^{2}  \right )^2}
                       \nonumber\\
        &\leq& \frac{(d+1)^{2}\,p_{\max}}{\underline{d}^2} \max\limits_{j=1,\dots,N_{vi}} \max\limits_{i\in\mathcal{N}_{j}}
            \left( \frac{ \sum\limits_{K\in\omega_{ij}} C_{\uD,h_{K}} \,|K|^{1/2}\,|\phi_{j}|_{1,K}}
            {\sum\limits_{K\in\omega_{ij}}  |\phi_{j}|_{1,K}^{2}  } \right)^{2}
                      \nonumber\\
        &\leq& \frac{(d+1)^{2}\,p_{\max}}{\underline{d}^2} \max\limits_{j=1,\dots,N_{vi}} \max\limits_{i\in\mathcal{N}_{j}}
            \left( \max\limits_{K\in\omega_{ij}} \frac{  C_{\uD,h_{K}} \,|K|^{1/2}\,|\phi_{j}|_{1,K}}
            {  |\phi_{j}|_{1,K}^{2}  } \right)^{2}
                      \nonumber\\
        &\leq& \frac{(d+1)^{2}\,p_{\max}}{\underline{d}^2} \max\limits_{j=1,\dots,N_{vi}} \max\limits_{i\in\mathcal{N}_{j}}
            \left( \max\limits_{K\in\omega_{ij}} \frac{ C_{\uD,h_{K}} \,|K|^{1/2}}
            { \frac{1}{h_{K}}\, |K|^{1/2} } \right)^{2}
                      \nonumber\\
        & = & \frac{(d+1)^{2}\,p_{\max}}{\underline{d}^2} \left( \max\limits_{K\in\mathcal{T}_{h}}
            \frac{ C_{\uD,h_{K}} \,|K|^{1/2}}
            { \frac{1}{h_{K}}\, |K|^{1/2} }  \right)^2            \nonumber\\
        &\leq& (d+1)^2C_{0}^2 H_{h}^2,
\end{eqnarray}
where $C_0$ and $H_h$ are defined in (\ref{CDh-1}).
From this we have
\begin{eqnarray*}
\sup_{\uv\not=0}\frac{\|A_{\delta}\uv\|^{2}}{\|\uv\|^{2}}
        &\leq& \sup_{\uv\not=0}\frac{\|A_{\delta}\uv\|^{2}}{\|A_{FE}^{D}\uv\|^{2}}
\cdot \sup_{\uv\not=0}\frac{\|A_{FE}^{D}\uv\|^{2}}{\|\uv\|^{2}} \leq (d+1)^2 C_{0}^2 H_{h}^2 (\max_{j} a_{jj}^{FE})^{2}.
\end{eqnarray*}
Then,
\begin{eqnarray} \label{eq:sigma_max_A0}
\sup_{\uv\not=0}\frac{\|A_{FV}\uv\|^{2}}{\|\uv\|^{2}} & = & \sup_{\uv\not=0} \frac{\|A_{FE}\uv + A_{\delta}\uv\|^{2}}{\|\uv\|^{2}}  \nonumber\\
        &\leq& \sup_{\uv\not=0} \frac{\|A_{FE}\uv\|^{2} + 2\|A_{FE}\uv \| \,
            \|A_{\delta}\uv\|+\|A_{\delta}\uv\|^{2}}{\|\uv\|^{2}}   \nonumber\\
        &\leq&  \sup_{\uv\not=0} \frac{\|A_{FE}\uv\|^{2}}{\|\uv\|^{2}}    + 2\sup_{\uv\not=0}\frac{\|A_{FE}\uv \|}{\|\uv\|}
               \cdot \sup_{\uv\not=0}\frac{\|A_{\delta}\uv\|}{\|\uv\|} + \sup_{\uv\not=0}\frac{\|A_{\delta}\uv\|^{2}}{\|\uv\|^{2}}   \nonumber\\
        &\leq& (d+1)^2(\max_{j} a_{jj}^{FE})^{2}
                 + 2\, (d+1)\max_{j} a_{jj}^{FE}\cdot (d+1)C_{0}H_{h}\max_{j} a_{jj}^{FE} \nonumber\\
        &    &   + ((d+1)C_{0}H_{h})^2 (\max_{j} a_{jj}^{FE})^{2} \nonumber\\
        &\leq&  (1+C_{0}H_{h})^2  (d+1)^{2}(\max_{j} a_{jj}^{FE})^{2}, \nonumber
\end{eqnarray}
which gives
\begin{equation}
\sigma_{\max}(A_{FV}) \le (1+C_{0}H_{h})  (d+1) \max_{j} a_{jj}^{FE} .
\label{sigma-AFV-1}
\end{equation}

From \cite[Lemma~2.5]{HuKaLa2013} we have
\[
\max_{j} a_{jj}^{FE}  \le C_{\tilde{\nabla}} \max\limits_{j}
\sum_{K \in \omega_j} |K|\,\|(F_{K}^{'})^{-1} \uD_{K} (F_{K}^{'})^{-T}\|_{2}.
\]
Combining this with (\ref{sigma-AFV-1}) we obtain (\ref{eigFV_max}).

For (\ref{eigSFVS_max}),  from Lemma~\ref{lem:diff-1},
(\ref{eq:AFEv-1}), and (\ref{eq:AFEv-2}) we have
\[
\frac{|(a_{jj}^{FV}-a_{jj}^{FE})|}{|a_{jj}^{FE}|}
   \leq \underline{d}^{-1} H_{h},
\]
which implies
\begin{equation}
 |a_{jj}^{FE}| \le \frac{|a_{jj}^{FV}|}{1 - \underline{d}^{-1}H_{h} } .
\label{eq:ajjvsajj}
\end{equation}
Combining this with (\ref{eq:sigma_max_A0}), we get
\[
\sigma_{\max}(A_{FV}) \le \frac{(1+C_{0}H_{h})}{1 - \underline{d}^{-1}H_{h}}  (d+1) \max_{j} a_{jj}^{FV} .
\]
Applying the same procedure for any diagonal scaling $S=(s_{j})$ we can obtain
\[
\max_{j} (s_{j}^{-2}a^{FV}_{jj}) \leq \sigma_{\max}(S^{-1}A_{FV}S^{-1})
        \leq \frac{(1+C_{0}H_{h})}{( 1 - \underline{d}^{-1}H_{h} )} (d + 1)\max_{j} (s_{j}^{-2}a^{FV}_{jj}).
\]
For the Jacobi preconditioning we have $s_{j}^2 = a^{FV}_{jj}$, which gives estimate (\ref{eigSFVS_max}).
\end{proof}

\subsection{Smallest eigenvalue of $(A_{FV}+A_{FV}^{T})/2$}

\begin{lem} \label{lem:PD}
$A_{FV}$ and $(A_{FV}+A_{FV}^{T})/2$ are positive definite when the mesh is sufficiently fine so that
$H_{h} < \underline{d}$.
\end{lem}

\begin{proof}
From (\ref{D-1}) we have
\begin{eqnarray} \label{eq:uAFEu}
a(u^{h},u^{h}) = \sum_{K\in\mathcal{T}_{h}}\int_{K} \nabla u^{h}\cdot(\uD \nabla u^{h}) \ud \ux
                \geq \underline{d}\, |u^{h}|^{2}_{1,\Omega}.
\end{eqnarray}
Then, from Lemma~\ref{lem:diff-2} we have
\begin{eqnarray}
\label{ahuPihu}
a_{h} (u^{h},\Pi_{h}^{*}u^{h}) &\geq& a(u^{h},u^{h}) - |a_{h} (u^{h},\Pi_{h}^{*}u^{h}) -a(u^{h},u^{h})| \nonumber\\
&\geq& (\underline{d} - H_{h})  |u^{h}|_{1,\Omega}^{2}.
\end{eqnarray}
From Poincare's inequality, there exists a constant $\gamma>0$ such that
\begin{eqnarray}
\uu^{T}\frac{A_{FV}+A_{FV}^{T}}{2}\uu = \uu^{T}A_{FV}\uu = a_{h} (u^{h},\Pi_{h}^{*}u^{h})
\geq \gamma (\underline{d} - H_{h}) ||u^{h}||^{2}_{L^{2}(\Omega)}. \nonumber
\end{eqnarray}
Thus, $A_{FV}$ and $(A_{FV}+A_{FV}^{T})/2$ are positive definite when $\underline{d} > H_{h}$.
\end{proof}

\begin{thm}
\label{thm:eigFV_min}
Assume that the mesh is sufficiently fine so that $H_h < \underline{d}$.
The smallest eigenvalue of $(A_{FV}+A_{FV}^{T})/2$
for the linear FVEM approximation of BVP $(\ref{BVP-1})$ is bounded from below by
\begin{equation}
\lambda_{\min}(\frac{A_{FV}+A_{FV}^{T}}{2}) \geq \frac{C \, \underline{d}}{N} \div
\begin{cases}
1, & \mathrm{for}\quad d=1,    \\
(1-\underline{d}^{-1} H_{h}) (1+\ln(\frac{|\overline{K}|}{|K_{min}|})), &   \mathrm{for}\quad d=2,    \\
(1-\underline{d}^{-1} H_{h}) \Big(\frac{1}{N}\sum_{K\in\mathcal{T}_{h}} \Big( \frac{|\overline{K}|}{|K|}
\Big)^{\frac{d-2}{2}}  \Big)^{\frac{2}{d}}, &   \mathrm{for}\quad d\geq3 ,
\end{cases}
\label{eq:eigFV_min}
\end{equation}
where $|\overline{K}|=\frac{1}{N}|\Omega|$ is the average element size and $C$ is a constant independent of
the mesh and the diffusion matrix.
Moreover,
the smallest singular value of the diagonally (Jacobi) preconditioned stiffness matrix is bounded from below by
\begin{equation}
\lambda_{\min}(S^{-1}\frac{A_{FV}+A_{FV}^{T}}{2}S^{-1})\geq \frac{C}{N^{\frac 2 d}} \div
\begin{cases}
\left ( \frac{1}{N \underline{d}} \sum\limits_{K\in\mathcal{T}_{h}}\uD(\ux_K)\frac{|\overline{K}|}{|K|} \right ),
        & \mathrm{for}\quad d=1, \\
\left( \frac{1}{N(\underline{d}-H_{h})} \sum\limits_{K\in\mathcal{T}_{h}}|K|
\,\|(F_{K}^{'})^{-1} \uD_{K}(F_{K}^{'})^{-T}\|_{2} \right) & \\
\qquad \quad \cdot \left (1+ \Big| \ln\frac{\max\limits_{K\in\mathcal{T}_{h}}
\|(F_{K}^{'})^{-1} \uD_{K}(F_{K}^{'})^{-T}\|_{2}}{\sum\limits_{K\in\mathcal{T}_{h}}|K|
\,\|(F_{K}^{'})^{-1} \uD_{K} (F_{K}^{'})^{-T}\|_{2}} \Big| \right ),
            &   \mathrm{for}\quad d=2,  \\
\left( \frac{1}{N(\underline{d}-H_{h})^{\frac{d}{2}}} \sum\limits_{K\in\mathcal{T}_{h}}|K|
\,\|(F_{K}^{'})^{-1} \uD_{K}(F_{K}^{'})^{-T}\|_{2}^{\frac{d}{2}} \right)^{\frac{2}{d}}, & \mathrm{for}\quad d\ge 3 .
\end{cases}
\label{eq:eigFVscaling_min}
\end{equation}
\end{thm}

\begin{proof}
The proof of this theorem is similar to that of Lemma 5.1 of \cite{KaHuXu2012}
for linear finite element approximation. For completeness, we give the detail of the proof here.

As in \cite{KaHuXu2012}, we need to treat the cases with $d = 1$, $d = 2$, and
$d\ge 3$ separately since the proof is based on Sobolev's inequality
\cite[Theorem 7.10]{Gilbarg01} which
has different forms in these cases. In the following, the function $u^h\in U^{h}$ and
its vector form $\uu=(u_{1},\dots,u_{N_{vi}})^{\mathrm{T}}$ are used synonymously.

\vspace{0.5cm}
\textit{Case} $d=1$. In one dimension, it is known (e.g., see \cite{LCW2000}) that
\[
  \uu^{T} A_{FV} \uu = a_{h}(u^{h},\Pi_{h}^{*}u^{h}) \geq \underline{d}\, | u_{h} |_{1,\Omega}^2.
\]
From Sobolev's inequality and the equivalence of vector norms, we have
\begin{align*}
\uu^{T} \frac{A_{FV}+A_{FV}^{T}}{2} \uu
& =   \uu^{T} A_{FV} \uu
\geq \underline{d}\, | u^{h} |_{1,\Omega}^2
\geq \underline{d}\,C_S \, |\Omega|^{-1} \sup_{\Omega} | u^{h} |^2
\\
& =   \underline{d}\,C_S \, |\Omega|^{-1} \max_{j} u_{j}^{2}
\geq \underline{d}\,C_S \, |\Omega|^{-1} N^{-1} \uu^{T} \uu,
\end{align*}
where $C_S$ is the constant associated with Sobolev's inequality. Thus, we have
\[
\lambda_{\min}((A_{FV}+A_{FV}^{T})/2) \geq C\underline{d} N^{-1},
\]
which gives (\ref{eq:eigFV_min}) (with $d=1$).

With diagonal scaling, we have
\begin{equation}
\label{eq:uSASu1D}
 \uu^{T}S^{-1} \frac{A_{FV}+A_{FV}^{T}}{2} S^{-1}\uu \geq C \underline{d} \max_{j} s_{j}^{-2}u_{j}^{2}
                \geq C \, \underline{d} \, \frac{\sum_{j}u_{j}^{2}}{\sum_{j}s_{j}^{2}}
                \geq C \, \underline{d} \, \frac{ \uu^T \uu }{ \sum_{j}s_{j}^{2} } .
\end{equation}
In one dimension, $\left | \nabla \phi_{j} \right| = |K|^{-1}$ when restricted in $K$. From this and noticing that $\omega_j$
contains at most two elements, we have
\begin{align*}
\sum\limits_j s_{j}^{2} & = \sum\limits_j A^{FV}_{jj}
= \sum\limits_j  \sum_{K\in\omega_{j}} \, -\mathbf{n}_{\ux_K} \cdot (\uD(\ux_K)  \nabla\phi_{j})
\\
& = \sum\limits_j  \sum_{K\in\omega_{j}} \, \frac{\uD(\ux_K)}{|K|}
\le 2 \sum_{K\in\mathcal{T}_h} \frac{\uD(\ux_K)}{|K|} ,
\end{align*}
where $\ux_K$ denotes the centroid of $K$ and $\mathbf{n}_{\ux_K}$ is the outward normal vector from $P_j$ to $\ux_K$.
Substituting this into (\ref{eq:uSASu1D}) we get (\ref{eq:eigFVscaling_min}) (for $d=1$).

\vspace{0.5cm}
\textit{Case} $d=2$.
From the proof of Lemma~5.1 of \cite{KaHuXu2012}, we have
\begin{eqnarray}
\label{uh1norm-0}
|u^{h}|_{1,\Omega}^2 &\geq& C q^{-1}
        \left(\sum_{K\in\mathcal{T}_{h}} \alpha_K^{\frac{q}{q-2}} \right)^{-\frac{q-2}{q}}
        \left ( \sum_{j} u_{j}^2 \sum_{K\in \omega_j} \alpha_K |K|^{\frac{2}{q}}\right ),
\end{eqnarray}
where $\{ \alpha_K,\; K \in \mathcal{T}_h\}$ is an arbitrary set of not-all-zero nonnegative numbers
and $q>2$ is an arbitrary constant. Taking $\alpha_K = |K|^{-2/q}$ gives
\begin{eqnarray}
\label{uh1norm}
|u^{h}|_{1,\Omega}^2 &\geq& C q^{-1}
        \left(\sum_{K\in\mathcal{T}_{h}} |K|^{-\frac{2}{q-2}} \right)^{-\frac{q-2}{q}} \sum_{j} u_{j}^2 .
\end{eqnarray}
Then from (\ref{ahuPihu}) and the above inequality we have
\begin{eqnarray}
\uu^{\mathrm{T}}\frac{A_{FV}+A_{FV}^{T}}{2}\uu   & = & \uu^{\mathrm{T}}A_{FV}\uu
 = a_{h}(u^{h},\Pi_{h}^{*}u^{h}) \nonumber\\
   &\geq& (\underline{d}-H_{h})|u^{h}|_{1,\Omega}^2  \nonumber\\
   &\geq& C (\underline{d}-H_{h}) q^{-1}
        \left(\sum_{K\in\mathcal{T}_{h}} |K|^{-\frac{2}{q-2}} \right)^{-\frac{q-2}{q}} \sum_{j} u_{j}^2
        \nonumber\\
   &\geq& C (\underline{d}-H_{h}) q^{-1}
        \left( N |K_{\min}|^{-\frac{2}{q-2}} \right)^{- \frac{q-2}{q}}  \sum_{j} u_{j}^2 \nonumber\\
   &=&    C (\underline{d}-H_{h}) N^{-1}
        \left[ q^{-1} (N\,|K_{\min}|)^{\frac{2}{q}} \right] \sum_{j} u_{j}^2,
\label{eq:lambdamin}
\end{eqnarray}
where $K_{\min}$ denotes the element with the minimal area. The above bound can be maximized
for $q=\max\{2,|\ln(N\,|K_{\min}|)|\}$ (with $q=2$ being viewed as the limiting case $q\to 2^+$) with
\[
  q^{-1} (N\,|K_{\min}|)^{\frac{2}{q}} \geq \frac{C}{1+|\ln(N\,|K_{\min}|)|}.
\]
Substituting this into (\ref{eq:lambdamin}) and using the definition of the average element size,
we obtain (\ref{eq:eigFV_min}) (with $d = 2$).

With diagonal scaling, we have
\begin{eqnarray}
\uu^{T}S^{-1} \frac{A_{FV}+A_{FV}^{T}}{2} S^{-1}\uu \geq C (\underline{d}-H_{h}) \frac{1}{q}
                  \left(\sum_{K\in\mathcal{T}_{h}} \alpha_{K}^{\frac{q}{q-2}} \right)^{-\frac{q-2}{q}}
                  \left(\sum_{j} u_{j}^2 s_{j}^{-2} \sum_{K\in\omega_{j}}\alpha_{K}|K|^{\frac{2}{q}} \right) .  \nonumber
\end{eqnarray}

For the Jacobi preconditioning $s_{j}^{2}=a^{FV}_{jj}$. With letting $u^{h}=\phi_{j}$ in $(\ref{lem:diff-2-1})$, we have
\begin{eqnarray}
s_{j}^{2}   &=& a^{FV}_{jj}= a^{FE}_{jj}+(a^{FV}_{jj}-a^{FE}_{jj})   \nonumber\\
            &\geq&  a(\phi_{j},\phi_{j}) - |a_{h}(\phi_{j},\Pi_{h}^{*}\phi_{j}) - a(\phi_{j},\phi_{j})| \nonumber\\
            &\geq&  \sum_{K\in \omega_{j}} |K|\,|\nabla\phi_{j}|\,|\nabla\phi_{j}|\, \left(\frac{ \nabla\phi_{j} }{|\nabla\phi_{j}|}\cdot(\uD_{K}\frac{ \nabla\phi_{j} }{|\nabla\phi_{j}|}) - H_{h} \right). \nonumber
\end{eqnarray}
Take
\[
\alpha_{K} = |K|^{\frac{q-2}{q}} \sum_{i_{K}=1}^{d+1} \nabla\phi_{i_{K}} \cdot(\uD_{K}\nabla\phi_{i_{K}})
           = |K|^{\frac{q-2}{q}} \sum_{i_{K}=1}^{d+1} \hat{\nabla}\hat{\phi}_{i_{K}}
                \cdot( (F_{K}^{'})^{-1} \uD_{K} (F_{K}^{'})^{-T} \hat{\nabla}\hat{\phi}_{i_{K}}),
\]
where $\hat{\phi}_{i}$ is a linear basis function and $\hat{\nabla}$ is the gradient operator on
the reference element $\hat{K}$. It is not difficult to show that
\[
s_{j}^{-2} \sum_{K\in\omega_{j}}\alpha_{K}|K|^{\frac{2}{q}}  \geq 1, \qquad
\alpha_{K} \leq (d+1)C_{\hat{\phi}} \, |K|^{\frac{q-2}{q}} \|(F_{K}^{'})^{-1} \uD_{K} (F_{K}^{'})^{-T}\|_{2},
\]
where $C_{\hat{\phi}}=\max_{i_{K}=1,\dots,d+1}\|\hat{\nabla}\hat{\phi}_{i_{K}}\|^{2}$. With these and choosing the value for the index $q$ in a similar manner as for the case without scaling we obtain $(\ref{eq:eigFVscaling_min})$
(for $d = 2$).

\vspace{0.5cm}
\textit{Case} $d\geq3$. Following a similar procedure as for the $d=2$ case, we have
\[
  \uu^{T} \frac{A_{FV}+A_{FV}^{T}}{2} \uu
            \geq C (\underline{d}-H_{h}) \left( \sum_{k\in\mathcal{T}_{h}} \alpha_{K}^{\frac{d}{2}}\right)^{-\frac{2}{d}}
                \sum_{j} u_{j}^{2} \sum_{K\in\omega_{j}} \alpha_{K} |K|^{\frac{d-2}{d}}.
\]
Choosing $\alpha_{K}=|K|^{-\frac{d-2}{d}}$ gives
\[
  \uu^{T} \frac{A_{FV}+A_{FV}^{T}}{2} \uu
            = C (\underline{d}-H_{h}) \left( \sum_{k\in\mathcal{T}_{h}} |K|^{\frac{2-d}{2}} \right)^{-\frac{2}{d}}
                \sum_{j} u_{j}^{2}.
\]
The estimate (\ref{eq:eigFVscaling_min}) for $d \ge 3$
follows from this and the definition of the average element size.

The bound for the diagonally scaled stiffness matrix is obtained by choosing
\[
\alpha_{K} = |K|^{\frac{2}{d}} \sum_{i_{K}=1}^{d+1} \hat{\nabla}\hat{\phi}_{i_{K}}
                \cdot( (F_{K}^{'})^{-1} \uD_{K} (F_{K}^{'})^{-T} \hat{\nabla}\hat{\phi}_{i_{K}}).
\]
\end{proof}

\subsection{Condition number of the stiffness matrix}

\begin{thm}
\label{thm-cond}
The condition number of the stiffness matrix for the linear finite volume element approximation
of homogeneous BVP $(\ref{BVP-1})$ is bounded by
\begin{align}
{\kappa}(A_{FV}) \leq & C\,(1+C_{0}H_{h})\, N^{\frac 2 d} \cdot
\left (\frac{1}{\underline{d} \, N^{\frac{2-d}{d}}}
\max\limits_{j} \sum_{K \in \omega_j} |K|\,\|(F_{K}^{'})^{-1} \uD_{K} (F_{K}^{'})^{-T}\|_{2}\right )
\notag \\
& \qquad \qquad \times
\begin{cases}
1, & \mathrm{for}\quad d=1,    \\
(1-\underline{d}^{-1} H_{h}) (1+\ln(\frac{|\overline{K}|}{|K_{min}|})), &   \mathrm{for}\quad d=2,    \\
(1-\underline{d}^{-1} H_{h}) \Big(\frac{1}{N}\sum_{K\in\mathcal{T}_{h}} \Big( \frac{|\overline{K}|}{|K|}
\Big)^{\frac{d-2}{2}}  \Big)^{\frac{2}{d}}, &   \mathrm{for}\quad d\geq3 ,
\end{cases}
\label{cond-1}
\end{align}
where $H_{h}=\max\limits_{K\in\mathcal{T}_{h}} \ud^2 ( h_{K}^{2}|\uD|_{2,\infty,K} + h_{K}|\uD|_{1,\infty,K})$
and $|\overline{K}|=\frac{1}{N}|\Omega|$ is the average element size.
With the diagonally (Jacobi) preconditioning, the ``condition number'' of the stiffness matrix is bounded by
\begin{align}
& {\kappa}(S^{-1}A_{FV}S^{-1}) \le
\notag \\
& \quad \frac{C\,(1+C_{0}H_{h})\, N^{\frac 2 d}}{(1-\underline{d}^{-1}H_{h})} \times
\begin{cases}
\left ( \frac{1}{N \underline{d}} \sum\limits_{K\in\mathcal{T}_{h}}\uD(\ux_K)\frac{|\overline{K}|}{|K|} \right ),
        & \mathrm{for}\quad d=1, \\
\left( \frac{1}{N(\underline{d}-H_{h})} \sum\limits_{K\in\mathcal{T}_{h}}|K|
\,\|(F_{K}^{'})^{-1} \uD_{K}(F_{K}^{'})^{-T}\|_{2} \right) & \\
\qquad \quad \cdot \left (1+ \Big| \ln\frac{\max\limits_{K\in\mathcal{T}_{h}}
\|(F_{K}^{'})^{-1} \uD_{K}(F_{K}^{'})^{-T}\|_{2}}{\sum\limits_{K\in\mathcal{T}_{h}}|K|
\,\|(F_{K}^{'})^{-1} \uD_{K} (F_{K}^{'})^{-T}\|_{2}} \Big| \right ),
            &   \mathrm{for}\quad d=2,  \\
\left( \frac{1}{N(\underline{d}-H_{h})^{\frac{d}{2}}} \sum\limits_{K\in\mathcal{T}_{h}}|K|
\,\|(F_{K}^{'})^{-1} \uD_{K}(F_{K}^{'})^{-T}\|_{2}^{\frac{d}{2}} \right)^{\frac{2}{d}}, & \mathrm{for}\quad d\ge 3 .
\end{cases}
\label{cond-2}
\end{align}
\end{thm}

\begin{proof} The conclusions follow from Theorems~\ref{thm:eigFV_max} and \ref{thm:eigFV_min}.
\end{proof}

The upper bounds in the above theorem also show the effects of the interplay between the mesh geometry and
the diffusion matrix. To see this, we consider $\uD^{-1}$-uniform meshes (a special case
of $\mathbb{M}$-uniform meshes) that are defined essentially as uniform meshes in the metric
specified by $\uD^{-1}$. It is known (e.g., see \cite{HR11}) that
a $\uD^{-1}$-uniform mesh $\mathcal{T}_{h}$ satisfies
\begin{equation}
(F_{K}^{'})^{-1} \uD_{K}(F_{K}^{'})^{-T} = h_{\uD^{-1}}^{-2} I, \quad \forall K \in \mathcal{T}_{h},
\label{M-uniform-1}
\end{equation}
where $h_{\uD^{-1}}$ is the average element size in metric $\uD^{-1}$, i.e.,
\[
h_{\uD^{-1}} = \left ( \frac{1}{N} \sum\limits_{K \in \mathcal{T}_{h}} |K| \,   \det(\uD_{K})^{-\frac 1 2}
\right )^{\frac{1}{d}}.
\]
From (\ref{D-1}) it is not difficult to see
\[
\frac{|\Omega|^{\frac 1 d}}{\sqrt{\overline{d}}} \le N^{\frac 1 d} h_{\uD^{-1}} \le
\frac{|\Omega|^{\frac 1 d}}{\sqrt{\underline{d}}} .
\]
Then, for a $\uD^{-1}$-uniform mesh, combining (\ref{M-uniform-1}) with the above theorem we have,
\begin{align}
{\kappa}(A_{FV}) \leq C\,(1+C_{0}H_{h})\, N^{\frac 2 d} \cdot
\left (N  \max\limits_{j} |\omega_j| \right ) \times
\begin{cases}
1, & \mathrm{for}\quad d=1,    \\
(1-\underline{d}^{-1} H_{h}) (1+\ln(\frac{|\overline{K}|}{|K_{min}|})), &   \mathrm{for}\quad d=2,    \\
(1-\underline{d}^{-1} H_{h}) \Big(\frac{1}{N}\sum_{K\in\mathcal{T}_{h}} \Big( \frac{|\overline{K}|}{|K|}
\Big)^{\frac{d-2}{2}}  \Big)^{\frac{2}{d}}, &   \mathrm{for}\quad d\geq3 ,
\end{cases}
\label{cond-3}
\end{align}
\begin{align}
{\kappa}(S^{-1}A_{FV}S^{-1}) \le \frac{C\,(1+C_{0}H_{h})^2\, N^{\frac 2 d}}{(1-\underline{d}^{-1}H_{h})^2} ,
\label{cond-4}
\end{align}
where $C$ is a constant which depends on $\uD$ but not on the mesh.
From (\ref{cond-3}) we can see that the mesh nonuniformity (in the Euclidean metric) can still
have significant effects on the conditioning of the stiffness matrix even for $\uD^{-1}$-uniform meshes.
Since a mesh cannot in general be uniform in the Euclidean metric and the metric $\uD^{-1}$ simultaneously,
mesh nonuniformity will have effects on the conditioning of the stiffness matrix.
On the other hand, the situation is different for Jacobi preconditioning.
The estimate (\ref{cond-4}) shows that the effects of mesh nonuniformity in the Euclidean metric
is totally eliminated by the preconditioning. In fact, the bound is almost the same as that
for the Laplace operator on a uniform mesh.

The above analysis is consistent with those of
\cite{Huang2014a,HuKaLa2013,KaHu2013b,KaHuXu2012}
for linear finite element discretization.
It also shows the importance of using Jacobi preconditioning and having a mesh that is uniform in the metric
specified by the inverse of the diffusion matrix.

\section{Conditioning of the mass matrix}
\label{SEC:mass}

In this section we discuss the mathematical properties for the mass matrix.
Although this is a topic not directly related to FVEM solution of boundary value problems,
it is useful for FVEM solution of time department and eigenvalue problems; e.g., see
\cite{Huang2014a,HuKaLa2013} for finite element discretization. Moreover,
it is theoretically interesting to know how the interplay between the mesh geometry and
the diffusion matrix affects the conditioning of the mass matrix.

The entices of the mass matrix $M=(M_{ij})$ are given by
\begin{eqnarray}
\label{eq:Mass-1}
M_{ij}  =  \int_{K_{P_{i}}^{*}} \phi_j \ud \ux
        = \sum_{K\in \omega_{ij}}\, |K|\, \int_{K\cap K_{P_{i}}^{*}} \phi_j \ud \ux
        = \left\{
                \begin{array}{lr}
                    m_{1} \,|\omega_{i}|, & i=j\\
                    m_{2} \, |\omega_{ij}|, & i\not=j
                \end{array}\right.,
                \quad i, j = 1, ..., N_{vi},
\end{eqnarray}
where $m_1$ and $m_2$ are given by (see Appendix A for the derivation)
\begin{equation}
m_{1} = \frac{1}{(d+1)^3} \left( 1+(d+1)\sum_{i=1}^{d}\frac{1}{i} \right),  \qquad
m_{2} = \frac{1}{d(d+1)^3} \left( d^2+2d-(d+1)\sum_{i=1}^{d}\frac{1}{i} \right).
\label{m1-m2}
\end{equation}
For $d = 1$, 2, and 3, we have
\[
m_{1} = \begin{cases}
 3/8, & d=1,\\
 11/54, & d=2,\\
 25/192, & d=3,
\end{cases} \qquad
m_{2} = \begin{cases}
 1/8, & d=1,\\
 7/108, & d=2,\\
 23/576, & d=3 .
\end{cases}
\]
Obviously, $M$ is symmetric. Moreover, denote the local mass matrix on $K$ by $M_K$
and that on $\hat{K}$ by $M_{\hat{K}}$. Then,
\[
(M_{\hat{K}})_{ij} = \int_{ \hat{K}_{P_{i}}^{*}\cap \hat{K} } \hat{\phi}_j \ud \boldsymbol{\xi} =
\begin{cases}
 m_{1}, & i=j,\\
 m_{2}, & i\not=j,
\end{cases}\qquad   i,j=1,\dots,d+1.
\]
It can also be shown that
\begin{eqnarray}
\label{eq:eigM}
(m_{1}-m_{2})I =  \frac{1}{(d+1)^2} \left( (d+1)\sum_{i=1}^{d}\frac{1}{i} - d\right) I
\leq M_{\hat{K}} \leq (m_1 + d m_2)I  = \frac{1}{d +1} I .
\end{eqnarray}
Then, for any vector $\uu$, letting $\boldsymbol{u}_{K}$ be the restriction
of the vector $\boldsymbol{u}$ on $K$ we have
\begin{eqnarray}  \label{eq:Masspositive}
\uu^{T}M\uu &=&     \sum_{K\in \mathcal{T}_{h}} \uu_{K}^{T}M_{K}\uu_{K}
             =      \sum_{K\in \mathcal{T}_{h}} |K|\, \uu_{K}^{T}M_{\hat{K}}\uu_{K}   \nonumber\\
            &\geq&  \sum_{K\in \mathcal{T}_{h}} (m_{1}-m_{2})\,|K|\, \|\uu_{K}\|^{2}
             =      \sum_{i} \left(  (m_{1}-m_{2})\, u_{i}^{2} \sum_{K\in\omega_{i}} |K| \right) \nonumber\\
            &\geq&     (m_{1}-m_{2})\,  \|\uu\|^{2}\, \min\limits_{i} |\omega_{i}|.
\end{eqnarray}
Thus, $M$ is also positive definite.

\subsection{Condition number of the mass matrix}

\begin{thm}
\label{thm:kappaM}
The condition number of the mass matrix for the linear FVEM on a simplicial mesh is bounded by
\begin{eqnarray}
 \frac{|\omega_{\max}|}{|\omega_{\min}|}
\leq \kappa(M) \leq \frac{1}{(d+1)(m_{1}-m_2)}
\frac{|\omega_{\max}|}{|\omega_{\min}|} .
\label{eq:kappaM}
\end{eqnarray}
\end{thm}

\begin{proof}
From (\ref{eq:eigM}), we have
\begin{eqnarray}
   \boldsymbol{u}^{T} M \boldsymbol{u}
   &=&\sum\limits_{K\in\mathcal{T}_{h}} \boldsymbol{u}_{K}^{T} M_{K} \boldsymbol{u}_{K}
       =\sum\limits_{K\in\mathcal{T}_{h}} |K| \boldsymbol{u}_{K}^{T} M_{\hat{K}} \boldsymbol{u}_{K}  \nonumber \\
   &\leq&  (m_1 + d \, m_2) \sum\limits_{K\in\mathcal{T}_{h}} |K| \|\boldsymbol{u}_{K}\|_{2}^{2} \nonumber \\
   &=&  \frac{1}{d+1} \sum\limits_{K\in\mathcal{T}_{h}} |K| \|\boldsymbol{u}_{K}\|_{2}^{2}.  \nonumber
\end{eqnarray}
Rearranging the sum on the right-hand side according to the vertices and using (\ref{eq:Mass-1}), we get
\begin{equation}
   \boldsymbol{u}^{T} M \boldsymbol{u}
   \leq  \frac{1}{d+1} \sum\limits_{K\in\mathcal{T}_{h}} |K| \|\boldsymbol{u}_{K}\|_{2}^{2}
   = \frac{1}{d+1}\sum\limits_{i} u_{i}^{2}|\omega_{i}|
   = \frac{1}{m_{1} (d+1)} \sum\limits_{i} u_{i}^{2} M_{ii},
\label{eq:uTMu}
\end{equation}
which implies
\[
  \lambda_{\max}(M) \leq  \frac{1}{m_{1}(d+1)} \max\limits_{i} M_{ii}.
\]
Similarly, we have
\begin{equation}
   \lambda_{\min}(M) \geq  \frac{m_{1}-m_{2}}{m_{1}} \min\limits_{i} M_{ii}.
\label{eq:uTMu_lowerbound}
\end{equation}
Moreover, it is not difficult to see that
\[
   \lambda_{\max}(M) \geq  \max\limits_{i} M_{ii},
   \qquad
   \lambda_{\min}(M) \leq  \min\limits_{i} M_{ii}.
\]
Combining the above estimates gives rise to
\begin{eqnarray}
  \max\limits_{i} M_{ii} \leq \lambda_{\max}(M) \leq  \frac{1}{m_{1} (d+1)} \max\limits_{i} M_{ii}, \nonumber\\
  \frac{m_{1}-m_{2}}{m_{1}} \min\limits_{i} M_{ii} \leq \lambda_{\min}(M) \leq  \min\limits_{i} M_{ii},    \nonumber
\end{eqnarray}
which lead to
\[
\frac{\max\limits_{i} M_{ii}}{\min\limits_{i} M_{ii}} \le \kappa(M)
\le \frac{1}{(m_{1}-m_2)(d+1)} \frac{\max\limits_{i} M_{ii}}{\min\limits_{i} M_{ii}} .
\]
From (\ref{eq:Mass-1}) we obtain (\ref{eq:kappaM}).
\end{proof}

The theorem shows that $\kappa(M)=\mathcal{O}(1)$ when the mesh is uniform or close to being uniform.
However, when the mesh is nonuniform, the condition number of $M$ can be very large.

\subsection{Diagonal scaling for the mass matrix}

For any diagonal scaling $S=(s_{i})$, like Theorem~\ref{thm:kappaM} we can obtain
\begin{eqnarray}
\frac{\max\limits_{i} s_{i}^{-2} M_{ii}}{\min\limits_{i} s_{i}^{-2} M_{ii}}
 \leq \kappa(S^{-1}MS^{-1})
 \leq \frac{1}{(m_{1}-m_{2})(d+1)} \frac{\max\limits_{i} s_{i}^{-2} M_{ii}}{\min\limits_{i} s_{i}^{-2} M_{ii}}. \label{eq:kappaSMS}
\end{eqnarray}
For the Jacobi preconditioning $s_{i}^{2}=M_{ii}$, we have the following theorem.

\begin{thm} \label{thm:kappaSMS}
The condition number of the Jocobi preconditioned FVEM mass matrix with a simplicial mesh
has a mesh-independent bound,
\begin{eqnarray}
\kappa(S^{-1}MS^{-1}) \leq \frac{1}{(m_{1}-m_{2})(d+1)}. \nonumber 
\end{eqnarray}
\end{thm}

The results in Theorems~\ref{thm:kappaM} and \ref{thm:kappaSMS} are similar to those results for linear FEM;
e.g., see \cite{KaHuXu2012,Wathen87}.

\subsection{Diagonal and lump of the mass matrix}

\begin{lem}
\label{lemma_MD}
The linear FVEM mass matrix $M$ and its diagonal part $M_{D}$ satisfy
\begin{eqnarray}
 \frac{m_{1}-m_{2}}{m_{1}} M_{D} \leq M\leq  \frac{1}{m_{1}(d+1)} M_{D},
 \label{eq:MD}
\end{eqnarray}
where the less-than-or-equal sign is in the sense of semi negative definiteness.
\end{lem}

\begin{proof}
This follows from (\ref{eq:uTMu}) and (\ref{eq:uTMu_lowerbound}) directly.
\end{proof}

\begin{lem} \label{lemma:Miilump}
Let $M_{lump}$ be the lumped linear FVEM mass matrix defined through
\begin{eqnarray*}
  M_{ii,lump} = \int_{K^{*}_{P_{i}}}  \sum\limits_{j=1}^{N_{vi}} \phi_{j}(\ux) \ud \ux, \quad i=1,\dots,N_{vi}.
\end{eqnarray*}
Then
\begin{eqnarray}
  m_{1}\,|\omega_{i}| \leq M_{ii,lump} \leq \frac{|\omega_{i}|}{d+1}. \label{eq:Miilump}
\end{eqnarray}
\end{lem}

\begin{proof}
Since
\begin{eqnarray*}
  \phi_{i}(\ux) \leq  \sum_{j=1}^{N_{vi}}\phi_{j}(\ux) \leq 1.
\end{eqnarray*}
With (\ref{eq:Mass-1}), we have
\begin{eqnarray*}
  M_{ii,lump} \geq \int_{K^{*}_{P_{i}}} \phi_{i}(\ux) \ud \ux
  = M_{ii}=m_{1} |\omega_{i}|
\end{eqnarray*}
and
\begin{eqnarray*}
  M_{ii,lump} \leq \int_{K^{*}_{P_{i}}} 1 \ud \ux
  = \frac{|\omega_{i}|}{d+1} .
\end{eqnarray*}
\end{proof}

\begin{lem}
The linear FVEM mass matrix $M$ and the lumped mass matrix $M_{lump}$ satisfy
\begin{eqnarray}
  (d+1)(m_{1}-m_{2})M_{lump} \leq M \leq \frac{1}{m_{1} (d+1)}M_{lump}.  \nonumber
\end{eqnarray}
\end{lem}
\begin{proof}
Since $M_{D}\leq M_{lump}$, we get the upper bound directly from (\ref{eq:MD}). The lower bound in (\ref{eq:MD}) together with the upper bound in (\ref{eq:Miilump}) give the lower bound
\begin{eqnarray}
  M \geq \frac{m_{1}-m_{2}}{m_{1}} M_{D} = \frac{m_{1}-m_{2}}{m_{1}}m_{1}{\mathrm{diag}}(|\omega_{1}|,\dots,|\omega_{N_{vi}}|) \geq (d+1)(m_{1}-m_{2})M_{lump}.   \nonumber
\end{eqnarray}
\end{proof}

\section{Numerical examples}
\label{sec:numerical}

In this section we present numerical results for a selection of $d$-dimensional ($d=1,\, 2,\, 3$)
examples to illustrate the theoretical results obtained in the previous sections.
Note that all bounds on the smallest eigenvalue $\lambda_{min}((A_{FV}+A_{FV}^{T})/2)$
(cf. Theorem~\ref{thm:eigFV_min})
contain a constant $C$. We obtain its value by calibrating the bound with uniform meshes
through comparing the exact and estimated values. For the largest singular value
$\sigma_{\max}(A_{FV})$ we use explicit bounds (\ref{eigFV_max}) and (\ref{eigSFVS_max})
where analytical expressions are available for the constants.
Predefined meshes are used to demonstrate the influence of the number and shape of mesh elements
on the condition number of the stiffness matrix and to verify the improvement achieved with the diagonal scaling.
The first three examples are adopted from \cite{KaHuXu2012}.
The results presented here for these examples are comparable with those obtained
in \cite{KaHuXu2012} with a linear finite element discretization.

\begin{exam}
\label{exam_1}
This is a one-dimensional example with $\uD = 1 + \exp(x^5)$ and a mesh given by Chebyshev nodes
in the interval [0, 1],
\[
x_{i}=\frac{1}{2} \left( 1-\cos \frac{\pi(2i-1)}{2(N-1)} \right),\quad i=1,\dots,N-1.
\]
The exact condition number of the stiffness matrix and its estimates (\ref{cond-1}) and (\ref{cond-2}) are shown in Fig.~\ref{fig:Example1D}(a). The exact $\sigma_{\max}(A_{FV})$ and $\lambda_{min}((A_{FV}+A_{FV}^{T})/2)$
and their estimates (\ref{eigFV_max}) and (\ref{eq:eigFV_min}) are shown in Fig.~\ref{fig:Example1D}(b).
The results show that the estimates have the same asymptotic order as the corresponding
exact values as N increases. Moreover, they show that the Jacobian preconditioning has significant impacts
on the condition number. Not only is ${\kappa}(S^{-1}A_{FV}S^{-1})$ significantly lower than
${\kappa}(A_{FV})$ but also it has a lower order than the latter does as $N$ increases.
\qed
\end{exam}

\begin{figure}[!htbp]
\centering
\begin{subfigure}{0.45\textwidth}
\centering
\includegraphics[scale = 0.18]{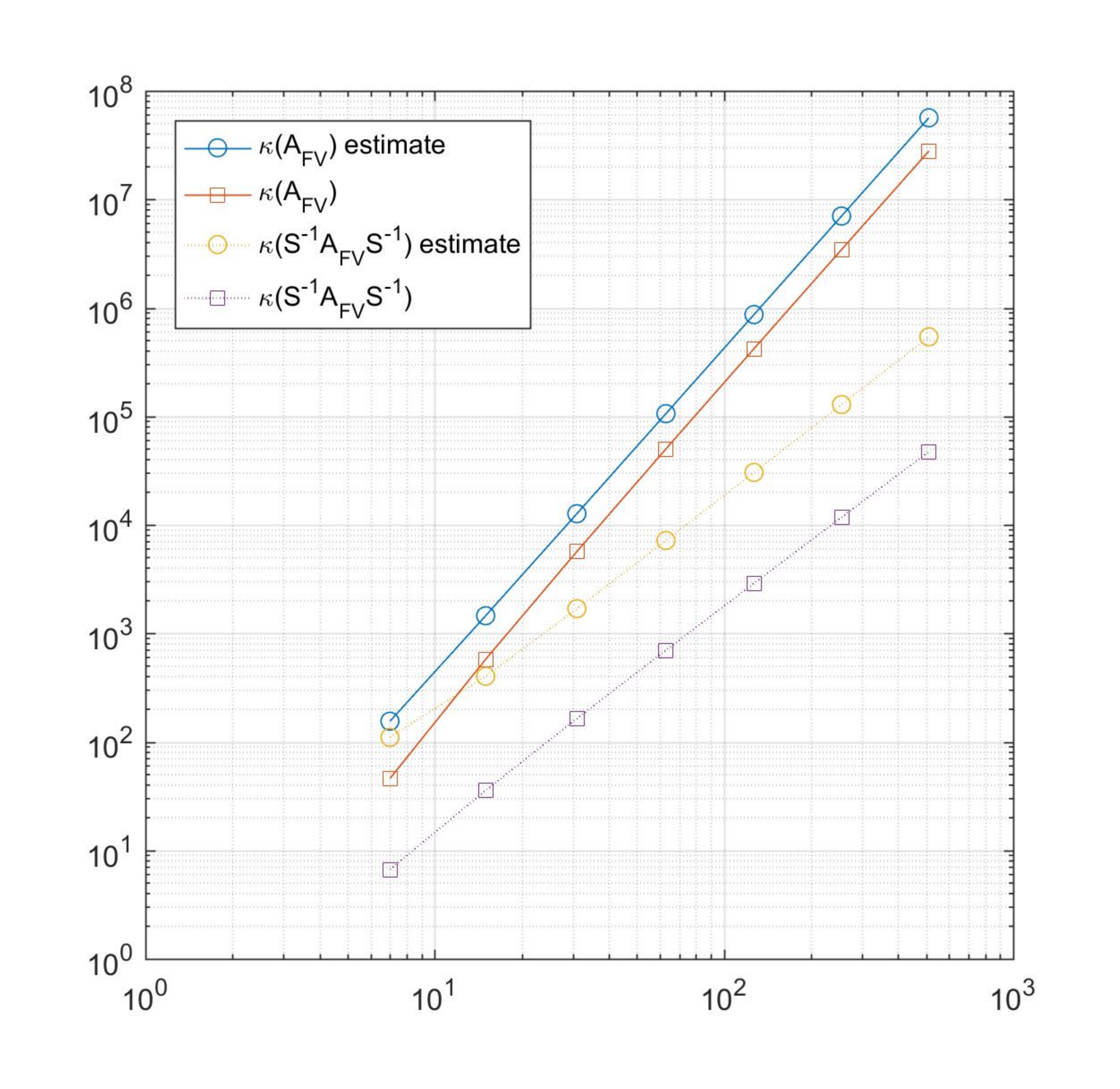}
\caption{ }
\end{subfigure}
\begin{subfigure}{0.45\textwidth}
\centering
\includegraphics[scale = 0.18]{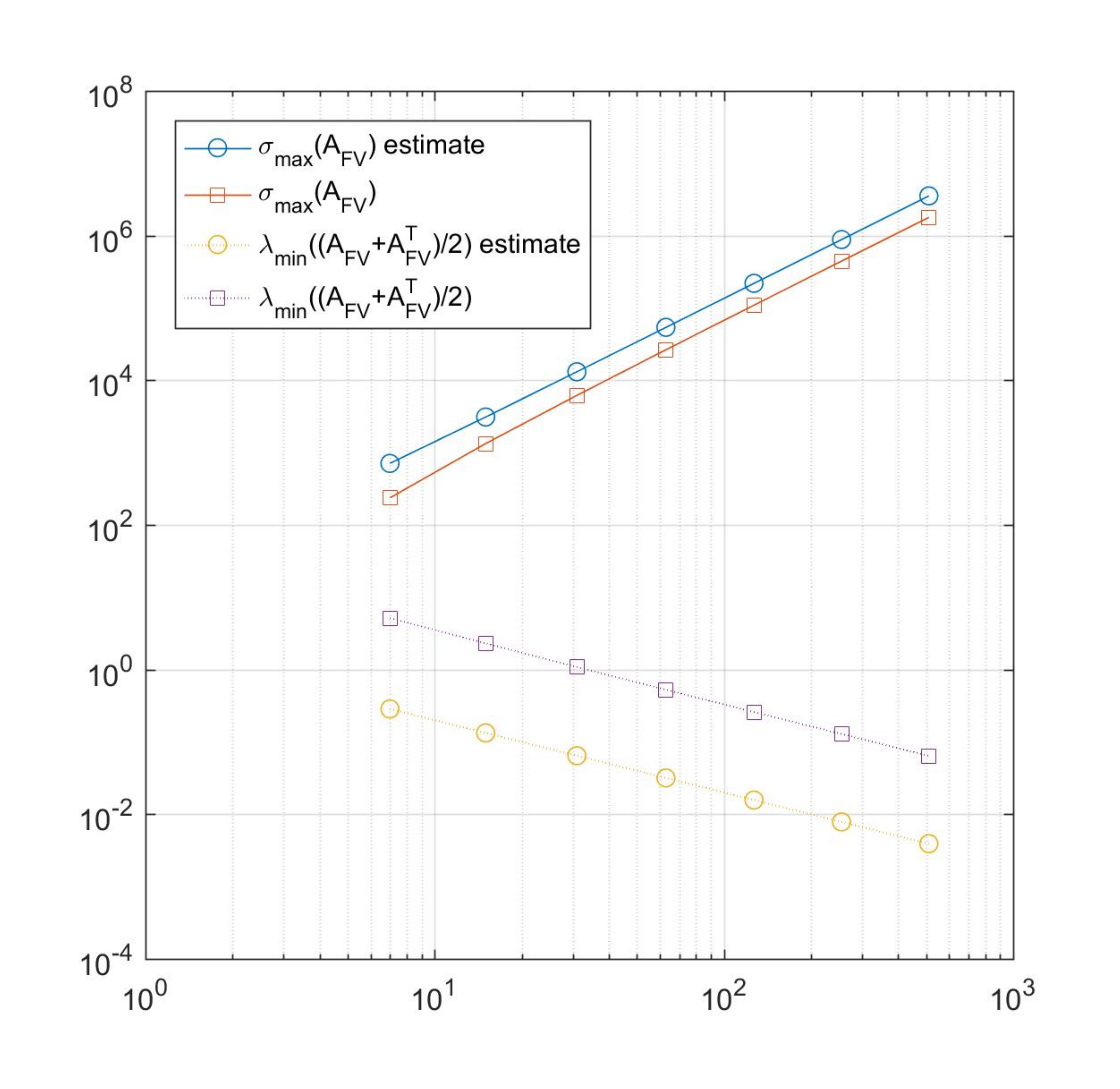}
\caption{}
\end{subfigure}
\caption{Example~\ref{exam_1}. Exact and estimated condition number (left) and exact and estimated greatest
singular value (eigenvalue) (right) of the stiffness matrix as a function of $N$ ($d = 1$).
}
\label{fig:Example1D}
\end{figure}

\begin{exam}
\label{exam_2}
In this two-dimensional example, $\uD = I$, $\Omega = (0, 1)\times (0, 1)$, and a mesh
(cf. Fig.~\ref{fig:Example2D}(a)) with $O(N^{1/2})$ skew elements and a maximum element aspect ratio
of $125:1$ are used. The condition number and its estimate are shown in Fig.~\ref{fig:Example2D}(b)
as functions of $N$. One can see that both the exact values and the estimates have the same
asymptotic order as $N$ increases. One can also see that the condition number with
scaling is significantly smaller than that without scaling and the asymptotic order of the former
is also smaller than that of the latter.
\qed
\end{exam}

\begin{figure}[!htbp]
\centering
\begin{subfigure}{0.45\textwidth}
\centering
\includegraphics[scale = 0.2]{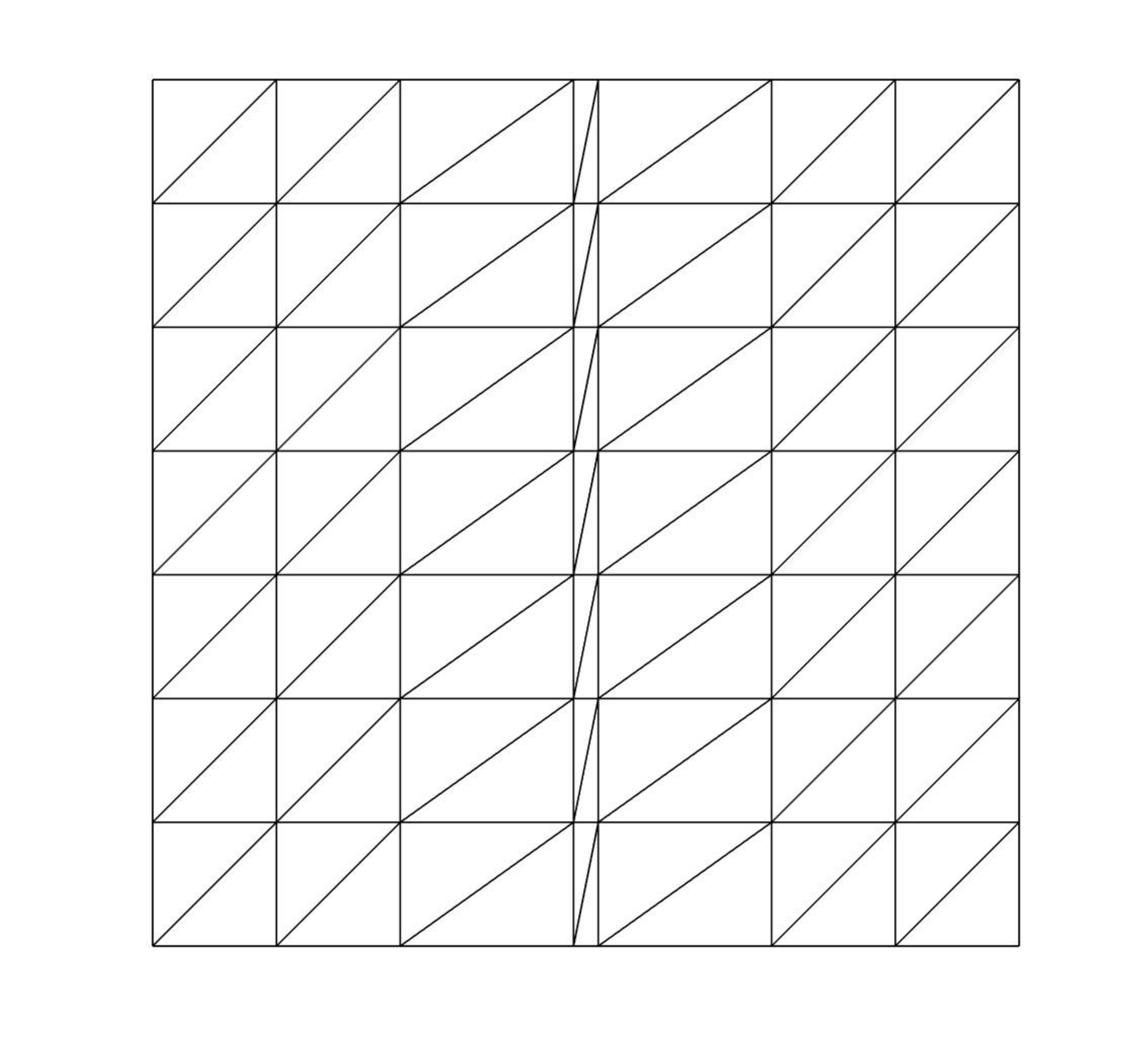}
\caption{  }
\end{subfigure}
\begin{subfigure}{0.45\textwidth}
\centering
\includegraphics[scale = 0.18]{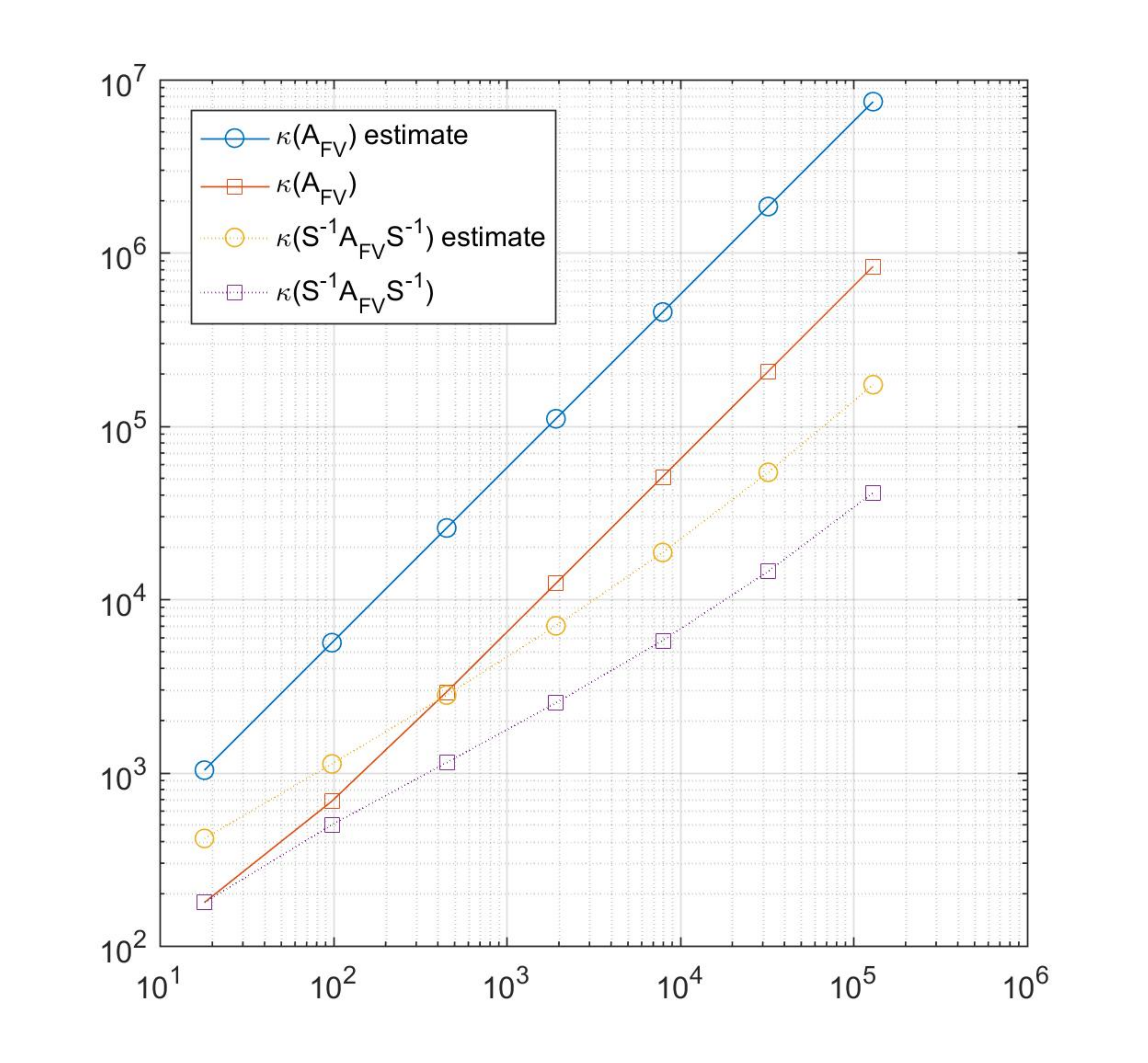}
\caption{ }
\end{subfigure}
\caption{Example~\ref{exam_2}. (a): A mesh example with a maximum element aspect ratio of 125:1
and (b): Exact and estimate condition numbers as functions of $N$.}
\label{fig:Example2D}
\end{figure}

\begin{figure}[!htbp]
\centering
\begin{subfigure}{0.45\textwidth}
\centering
\includegraphics[scale = 0.35]{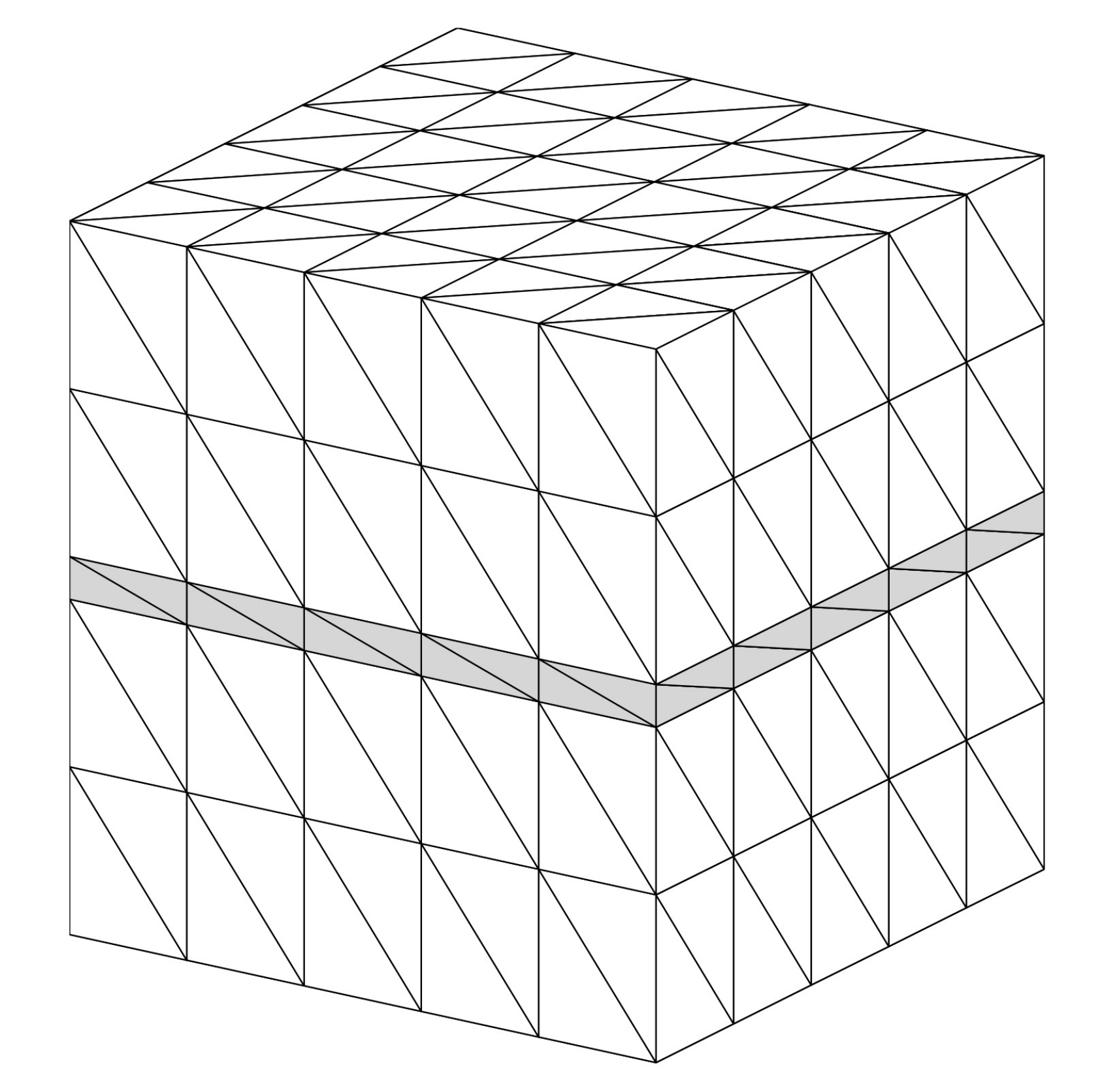}
\caption{ }
\end{subfigure}
\begin{subfigure}{0.45\textwidth}
\centering
\includegraphics[scale = 0.18]{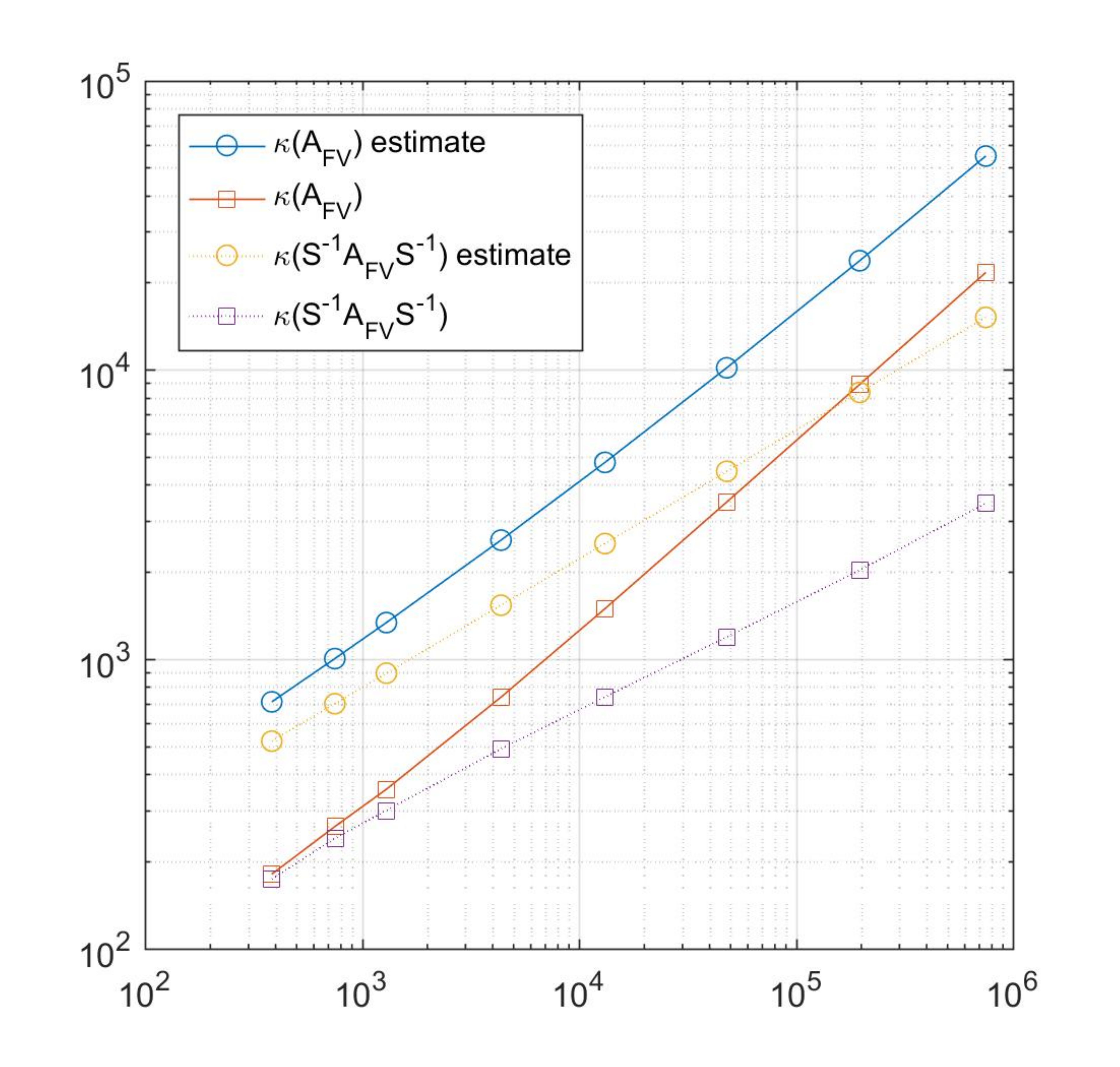}
\caption{ }
\end{subfigure}
\caption{Example~\ref{exam_3}.(a): A mesh example with a maximum element aspect ratio of 125:1
and (b): Exact and estimate condition numbers as functions of $N$.}
\label{fig:Example3D}
\end{figure}

\begin{exam}
\label{exam_3}
In this three-dimensional example, $\uD = I$, $\Omega$ is the unit cube, and a mesh
shown in Fig.~\ref{fig:Example3D}(a) and having $O(N^{2/3})$ skew elements with
a maximum aspect ratio of $125:1$ is used. The results are shown in Fig.~\ref{fig:Example3D}(b)
as $N$ increases. We can see that scaling not only reduces the condition number
significantly but also lowers the asymptotic order in $N$. Moreover, the bound (\ref{cond-1})
and $\kappa (A_{FV})$ have the same asymptotic order. However, the order of the bound (\ref{cond-2})
in $N$ is slightly higher than that of $\kappa (S^{-1} A_{FV}S^{-1})$.
Similar trends have been observed
for a linear finite element discretization in \cite{KaHuXu2012}.
\qed
\end{exam}

\begin{figure}[!htbp]
\centering
\begin{subfigure}{0.45\textwidth}
\centering
\includegraphics[scale = 0.3]{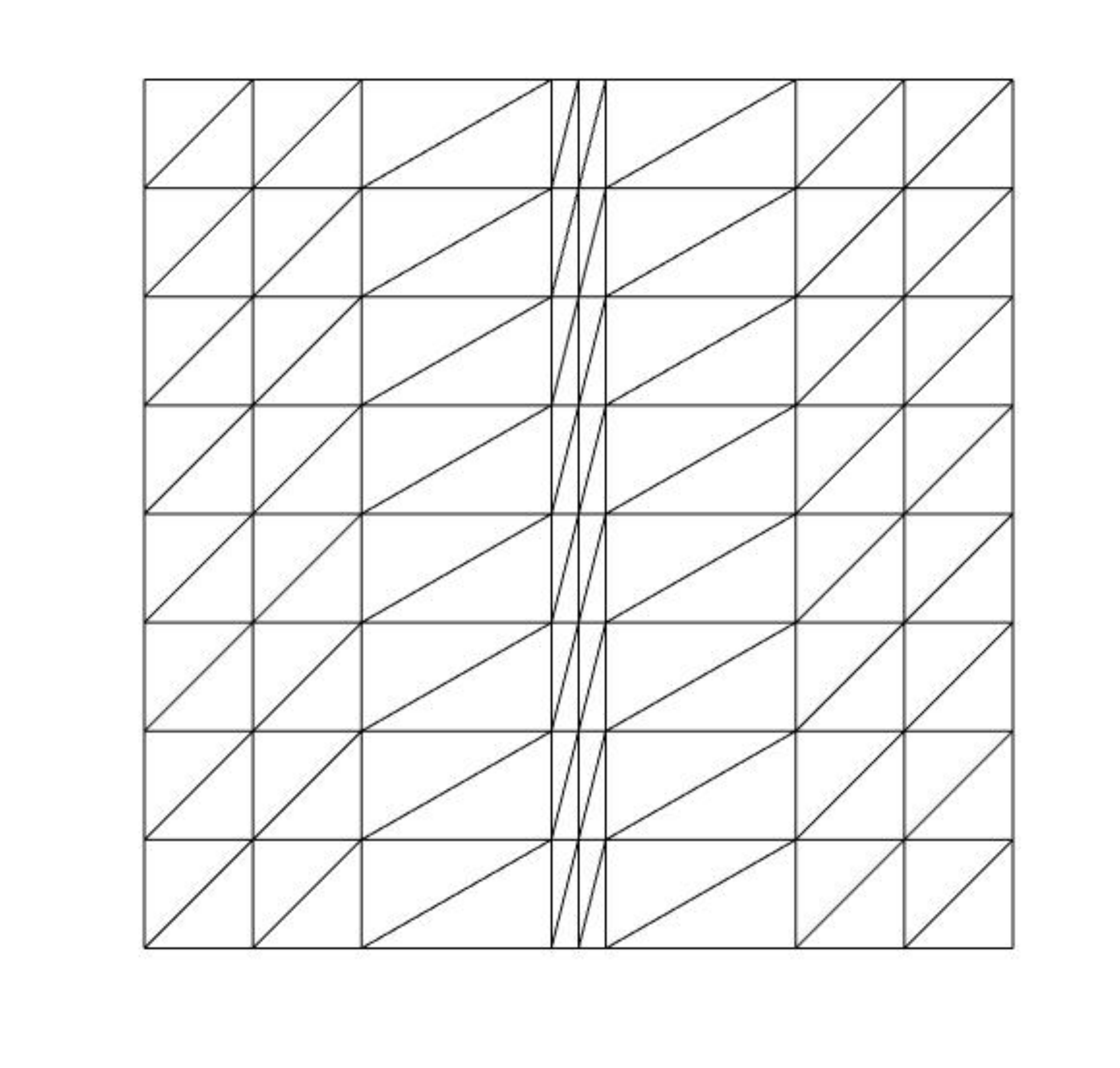}
\caption{ }
\end{subfigure}
\begin{subfigure}{0.45\textwidth}
\centering
\includegraphics[scale = 0.25]{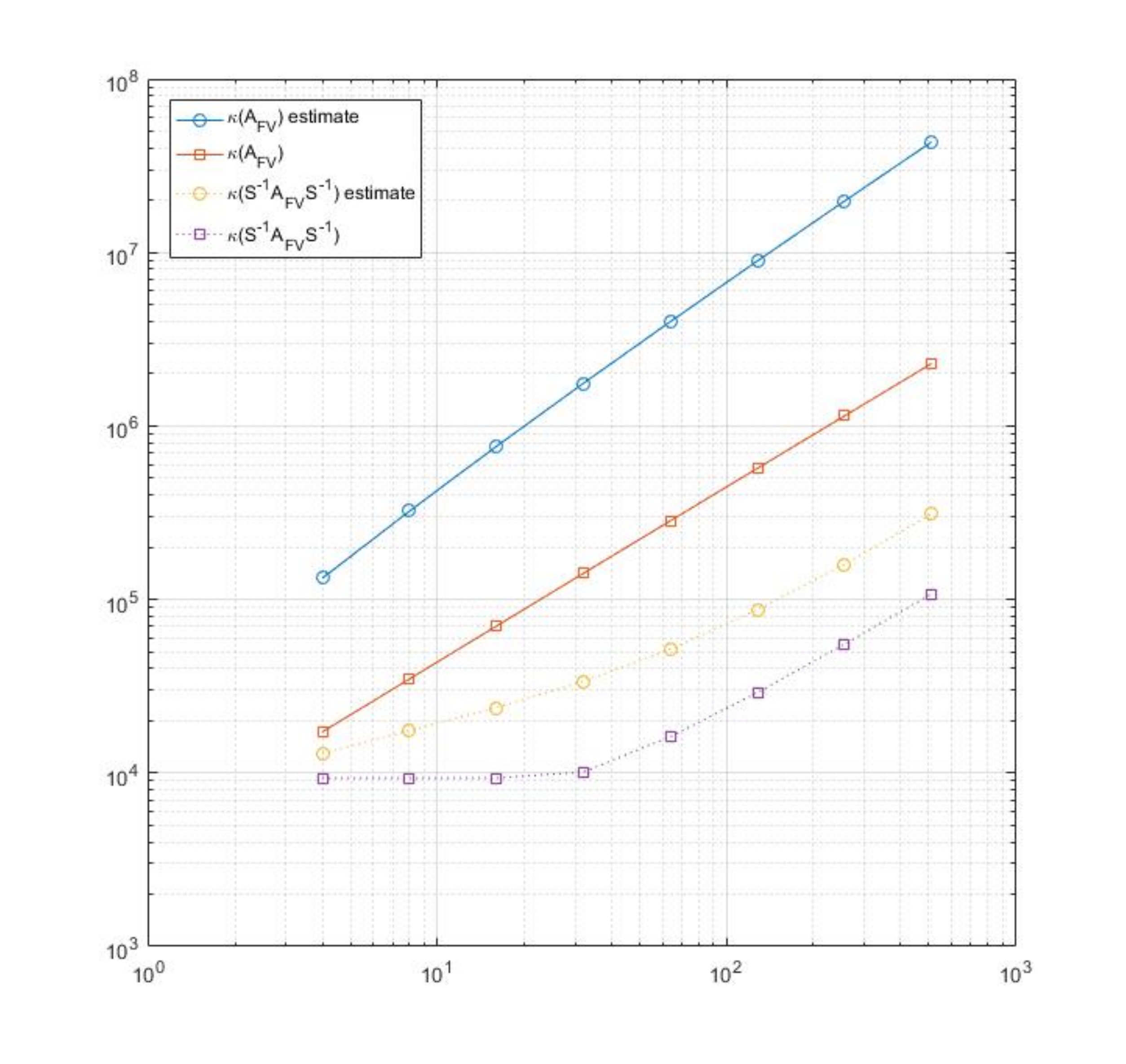}
\caption{ }
\end{subfigure}
\caption{Example~\ref{exam_4}: (a): The predefined meshes 
and (b): Exact and estimate condition numbers as functions of the maximum element aspect ratio
when the size of the mesh is fixed at $N=32258$.}
\label{fig:Example2D_rato}
\end{figure}

\begin{figure}[!htbp]
\centering
\begin{subfigure}{0.45\textwidth}
\centering
\includegraphics[scale = 0.4]{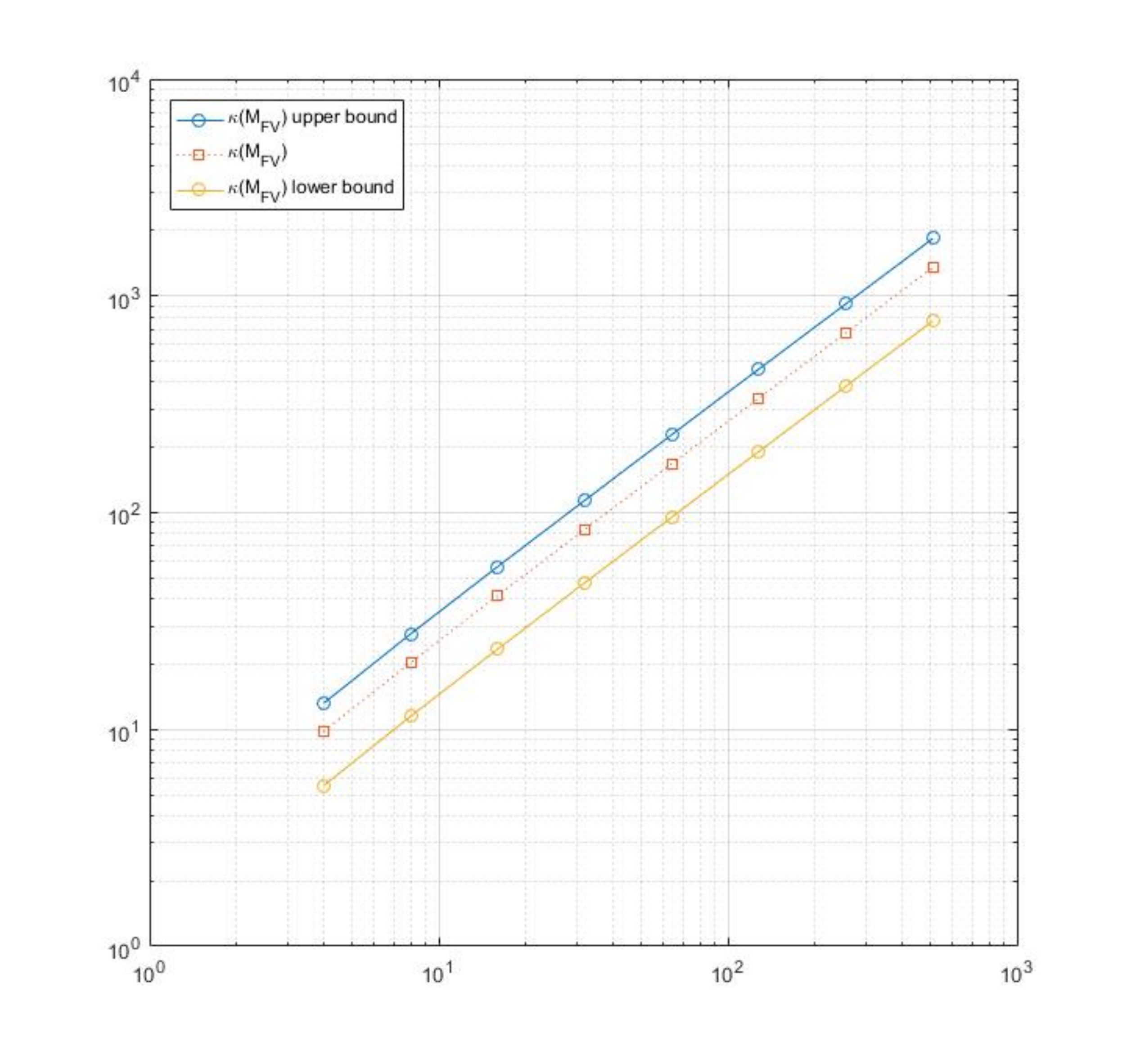}
\caption{ }
\end{subfigure}
\begin{subfigure}{0.45\textwidth}
\centering
\includegraphics[scale = 0.4]{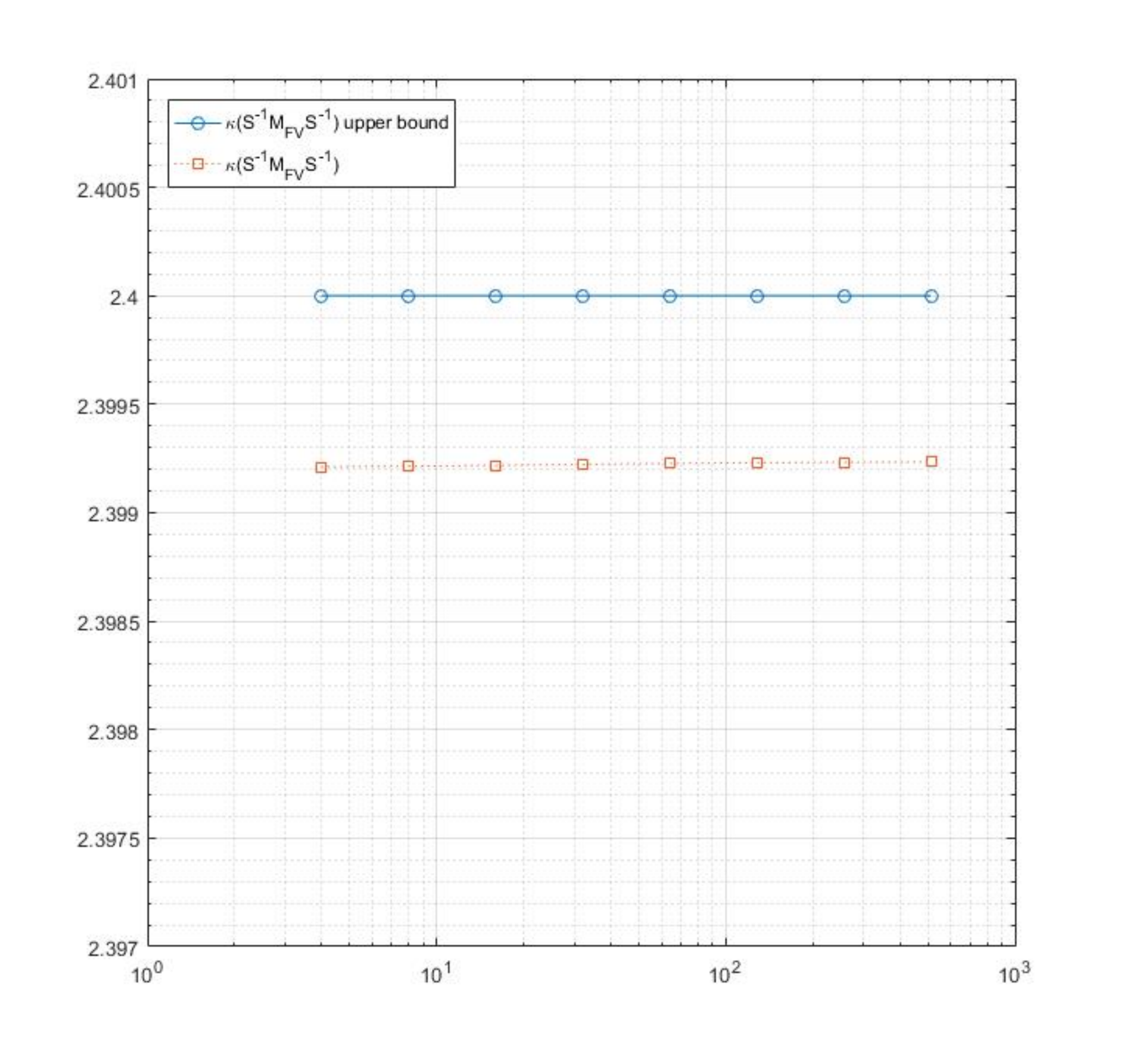}
\caption{ }
\end{subfigure}
\caption{Example~\ref{exam_4}: The condition of the mass matrix and its bounds are shown as
functions of the maximum element aspect ratio when the size of the mesh is fixed at $N=32258$.
(a) Without Jacobian preconditioning and (b) With Jacobian preconditioning.}
\label{fig:Example2D_rato_mass}
\end{figure}

\begin{exam}
\label{exam_4}
The setting of this example is essentially the same as that in Example~\ref{exam_2}
except that the size of the mesh is fixed at $N = 32258$ but its maximum aspect ratio of elements
increases and that the diffusion matrix is chosen as
\[
\uD = \begin{bmatrix} \cos(\theta) & - \sin (\theta) \\  \sin (\theta) & \cos(\theta) \end{bmatrix}
\begin{bmatrix} 1 & 0 \\ 0 & 0.01 \end{bmatrix}
\begin{bmatrix} \cos(\theta) &  \sin (\theta) \\  - \sin (\theta) & \cos(\theta) \end{bmatrix} ,
\]
where $\theta = \pi \sin(x) \cos(y)$.
An example of mesh and the condition number of the stiffness matrix and its estimate are shown
in Fig.~\ref{fig:Example2D_rato}. The results show that the condition number and its estimate
are essentially linear functions of the maximum element aspect ratio. Moreover, the condition number
is much smaller with scaling than without scaling.

The condition number and its bounds for the mass matrix are shown in Fig.~\ref{fig:Example2D_rato_mass}.
Without scaling, they increase linearly with the maximum element aspect ratio. On the contrary, they stay
constant when the Jacobian scaling is used. This is consistent with Theorem~\ref{thm:kappaSMS}.
\end{exam}

%
%
%

\section{Conclusions}
\label{SEC:conclusions}

In the previous sections we have studied the conditioning of the stiffness matrix $A_{FV}$
of the linear finite volume element discretization of the boundary value problem (\ref{BVP-1})
with general simplicial meshes.
Since $A_{FV}$ is nonsymmetric in general, we define its condition number (\ref{cond_0}) as
the ratio of the maximum singular value, $\sigma_{max}(A_{FV})$, to the minimum eigenvalue
of its symmetric part, $\lambda_{min}((A_{FV}+A_{FV}^{T})/2)$, in lieu of the convergence
of GMRES (cf. (\ref{Eisenstat-1})). The situations with and without Jacobian preconditioning
have been considered. An upper bound on the maximum singular
value and a lower bound on the minimum eigenvalue of the symmetric part have been obtained
in Theorems~\ref{thm:eigFV_max} and \ref{thm:eigFV_min}, respectively, and an upper bound
on the condition number has been obtained in Theorem~\ref{thm-cond}.

It is noted that those theoretical results have been obtained for a general diffusion matrix $\uD$ and
a sufficiently fine, arbitrary simplicial mesh in any dimension.
They not only provide a bound on the condition number of the stiffness matrix
but also shed light on the effects of the interplay between the diffusion matrix
and the mesh geometry. Particularly, the bounds reveal that without scaling, the condition number
is affected by the number of the elements $N$, the mesh nonuniformilty in the Euclidean metric, and
the mesh nonuniformilty in the metric specified by $\uD^{-1}$.
For meshes that are uniform in $\uD^{-1}$, the last factor will be eliminated but
the mesh nonuniformilty in the Euclidean metric still plays a role; see (\ref{cond-3}).
On the other hand, the analysis shows that the effects by the mesh nonuniformilty in the Euclidean metric
can be eliminated by scaling. For the situation with scaling and a $\uD^{-1}$-uniform mesh,
the condition number depends only on the number of the elements (cf. (\ref{cond-4})).
Numerical examples confirm the above analysis.

A similar analysis has been carried out for the mass matrix in \S\ref{SEC:mass}.
The main results are stated in Theorems~\ref{thm:kappaM} and \ref{thm:kappaSMS}.
They show that the condition number of the mass matrix for the linear
FVEM discretization depends only on the mesh nonuniformilty in the Euclidean metric
and scaling can effectively eliminate its effects.

It is remarked that the results and observations made in this work are comparable and consistent with
those in \cite{KaHuXu2012} for a linear finite element discretization of (\ref{BVP-1}).
The only noticeable difference is that the assumption of the mesh being sufficiently fine is needed
in the current analysis.
This is not surprising since FVEM generally does not preserve the symmetry of the underlying differential
operator. Moreover, when the mesh is sufficiently fine, roughly speaking,
both the FVEM and FEM discretizations are close to the differential operator and
thus should exhibit similar behaviors. In this spirit, it is expected that the analysis in this work
can be extended to higher-order FVEMs without major modifications; see \cite{HKL2013b}
for studies for higher-order FEMs.

\vspace{20pt}

\section*{Acknowledgments}
The work was supported in part by
the National Natural Science Foundation of China through grants 11701211 and 11371170, the China Postdoctoral Science Foundation through grant 2017M620106, the Joint Fund of the National Natural Science Foundation of China and the China Academy of Engineering Pysics (NASF) through grant U1630249, and
the Science Challenge Program (China) through grant JCKY2016212A502.
X.W. was supported by China Scholarship Council (CSC) under grant 201506170088
for his research visit to the University of Kansas from September of 2015 to September of 2016.
X.W. is thankful to the Department of Mathematics of the University of Kansas for the hospitality
during his visit.


\section*{Appendix A: The expressions for $m_1$ and $m_2$}

To obtain the values of $m_1$ and $m_2$ for general $d$ dimensions, we consider $K$ to be a right simplex
as shown in Fig.~\ref{fig:order1forDim_d} for two and three dimensions.
The dual element $K_{P_{i}}^{*}$ restricted to the primary element $K$ is a polyhedron with $2d$ faces
(see the polyhedron $P_{1}M_{1}M_{0}M_{2}$ in Fig.~\ref{fig:order1forDim_d}(a) and
the polyhedron $P_{1}M_{1}M_{2}M_{3}M_{4}M_{5}M_{6}M_{0}$ in Fig.~\ref{fig:order1forDim_d}(b)).
We now consider $d = 2$, $d = 3$, and a general $d$ case separately.

\begin{figure}[!htbp]
\centering
\begin{subfigure}{0.4\textwidth}
\centering
\includegraphics[scale = 0.5]{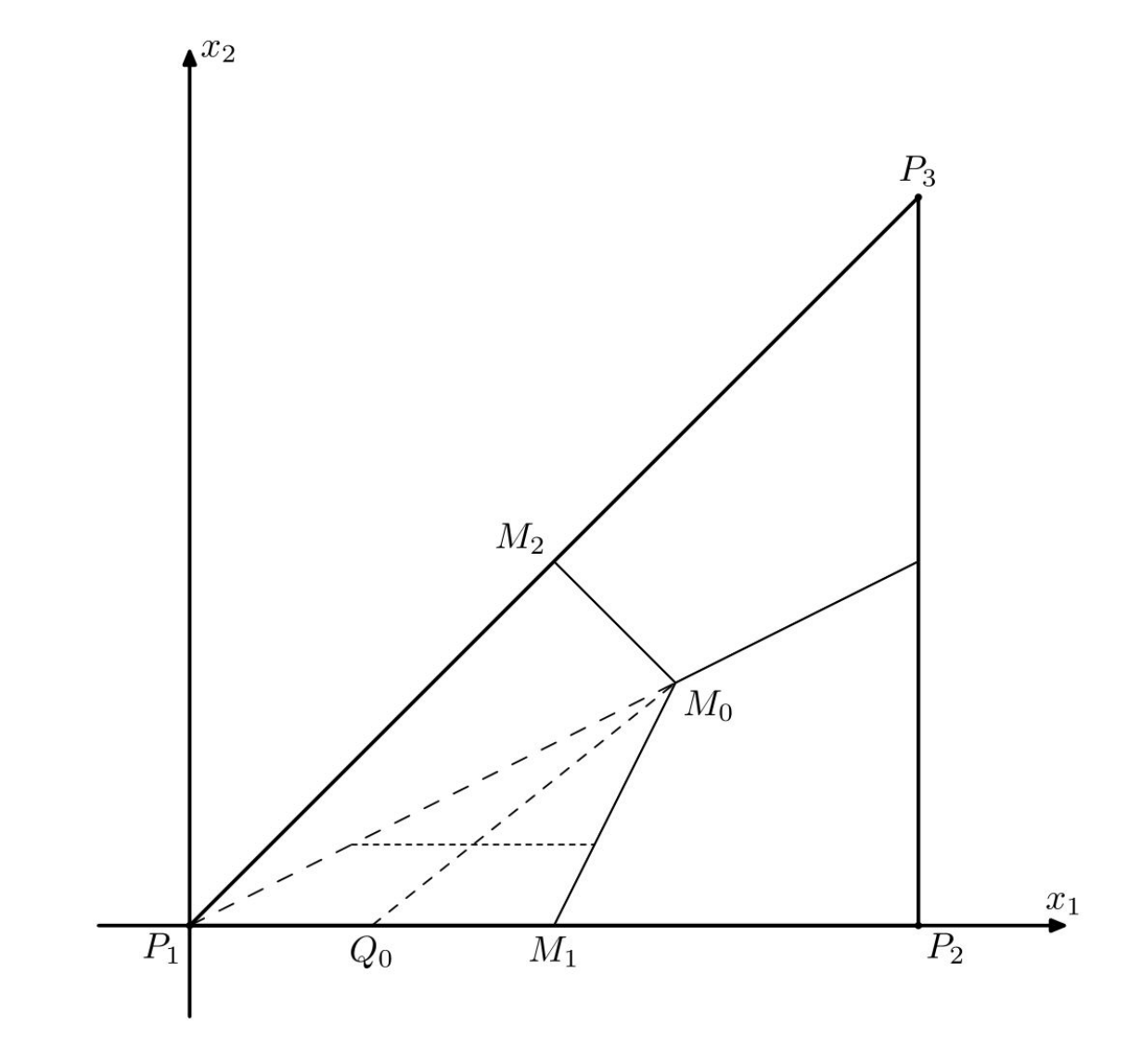}
\caption{}
\end{subfigure}
\begin{subfigure}{0.4\textwidth}
\centering
\includegraphics[scale = 0.5]{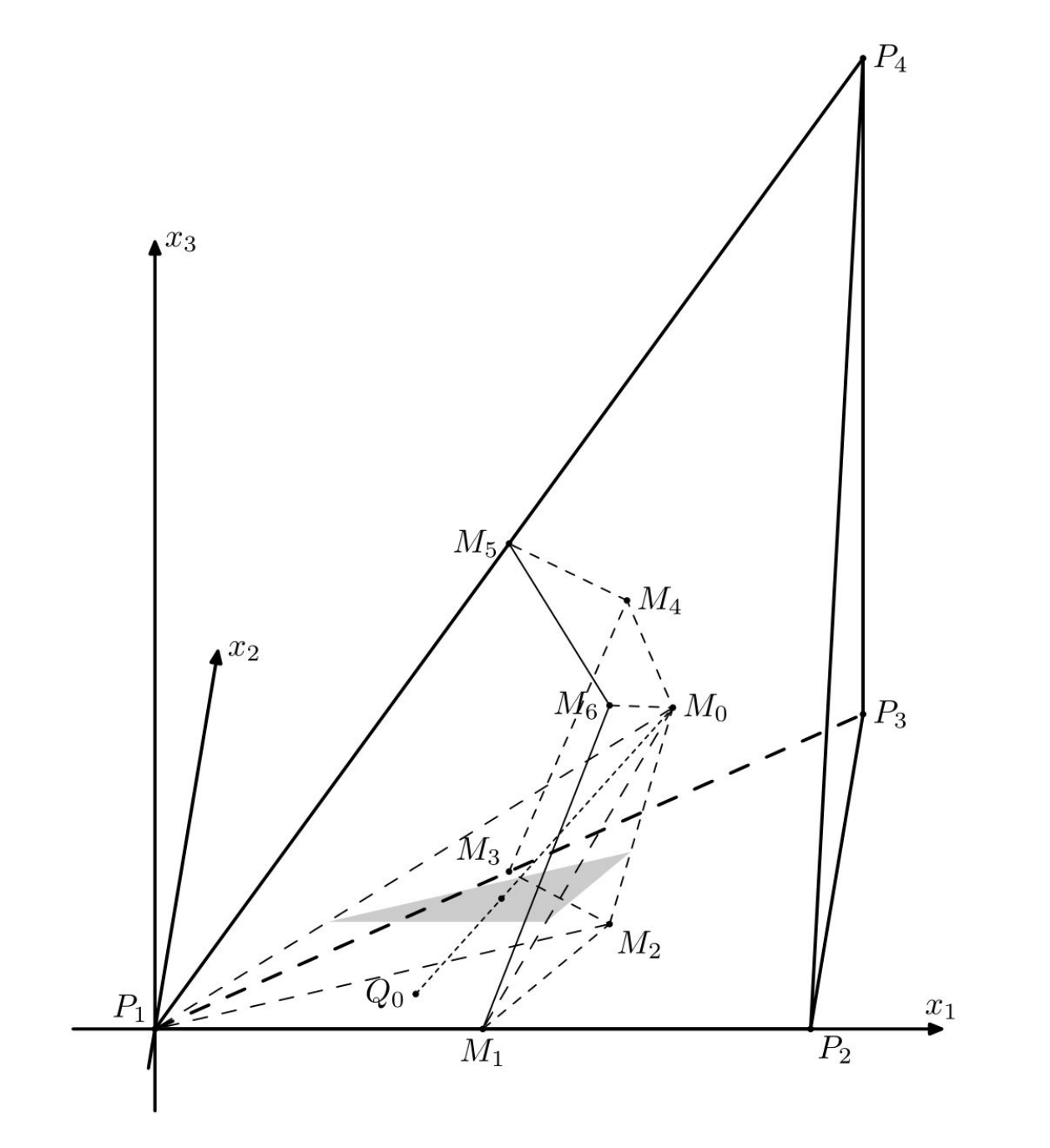}
\caption{}
\end{subfigure}
\caption{(a) A right triangle in two dimensions; (b) A right tetrahedron in three dimensions.}
\label{fig:order1forDim_d}
\end{figure}

\vspace{10pt}

{\em The $d=2$ case.} Consider the triangle $K=\triangle P_{1}P_{2}P_{3}$ in Fig.~\ref{fig:order1forDim_d}(a),
where $P_{1}=(0,0)$, $P_{2}=(1,0)$, and $P_{3}=(1,1)$. Denote the midpoints of $P_{1}P_{2}$
and $P_{1}P_{3}$ by $M_{1}$ and $M_{2}$, respectively, and the barycenter of $K$ by $M_{0}$.
Then $\phi_{P_{1}}|_{K}=1-x$. It is not difficult to see that
\begin{eqnarray}
\int_{K_{P_{i}}^{*}\cap K} \phi_{P_{1}} \ud \ux = 2\int_{K_{0}} \phi_{P_{1}} \ud \ux
= m_{1} |K| =\frac{1}{2} m_{1}.
\label{eq:m1Vsd2D_a}
\end{eqnarray}
Here, $K_{0}=\triangle P_{1}M_{1}M_{0}$. Since $\phi_{P_{1}}|_{K}$ is a linear function, the integral of $\phi_{P_{1}}$ on $P_{1}M_{1}$ equals $\phi_{P_{1}}(Q_{0})$ multiplied by the length of $P_{1}M_{1}$. Thus,
\begin{eqnarray}
&& \int_{K_{0}} \phi_{P_{1}} \ud \ux
\nonumber \\
&=& \int_{y(Q_{0})}^{y(M_{0})}
    \left( \frac{y-y(Q_{0})}{y(M_{0})-y(Q_{0})}(\phi_{P_{1}}(M_{0})-\phi_{P_{1}}(Q_{0})) + \phi_{P_{1}}(Q_{0}) \right)  \frac{y(M_{0})-y}{y(M_{0})-y(Q_{0})} |P_{1}M_{1}|  \,\ud y \nonumber \\
&=& \int_{0}^{\frac{1}{3}}
    \left( \frac{y}{(\frac{1}{3}-0)}( \, (1-\frac{2}{3})-(1-\frac{1}{4})) + (1-\frac{1}{4})  \right)
         \frac{(\frac{1}{3}-y)}{(\frac{1}{3}-0)} \frac{1}{2}  \,\ud y \nonumber \\
&=& \frac{11}{216}.  \nonumber
\end{eqnarray}
Combining this with (\ref{eq:m1Vsd2D_a}), we have
\[
m_{1} = 4\int_{K_{0}} \phi_{P_{1}} \ud \ux = \frac{11}{54}.
\]

On the other hand,
\begin{eqnarray}
|K_{P_{i}}^{*}\cap K|=\frac{1}{3}|K| =\int_{K_{P_{i}}^{*}\cap K} \sum\limits_{j=1,2,3} \phi_{P_{j}} \ud \ux = (m_{1}+2m_{2}) |K|.    \nonumber
\end{eqnarray}
Then,
\[
m_{2}=\frac{7}{108}.
\]

\vspace{10pt}

{\em The $d=3$ case.} Consider the tetrahedron $K=P_{1}P_{2}P_{3}P_{4}$ in Fig.~\ref{fig:order1forDim_d}(b),
where $P_{1}=(0,0,0)$, $P_{2}=(1,0,0)$, $P_{3}=(1,1,0)$, and $P_{4}=(1,1,1)$.
Denote the midpoints of $P_{1}P_{2}$, $P_{1}P_{3}$, and $P_{1}P_{4}$ by $M_{1}$, $M_{3}$, and $M_{5}$,
respectively, and the barycenters of the corresponding faces of $K$ by $M_{2}$, $M_{4}$, and $M_{6}$.
Let $M_{0}$ be the centroid of $K$ and $Q_{0}$ be the barycenter of $\triangle P_{1}M_{1}M_{2}$.
Then $\phi_{P_{1}}|_{K}=1-x$. It is not difficult to see that
\begin{eqnarray}
\frac{1}{6} m_{1}= m_{1} |K| =\int_{K_{P_{i}}^{*}\cap K} \phi_{P_{1}}
\ud \ux = 3\times ( 2 \int_{K_{0}} \phi_{P_{1}} \ud \ux )
.  \label{eq:m1Vsd3D_a}
\end{eqnarray}
Here, $K_{0}$ is the tetrahedron formed by the vertices $P_{1}$, $M_{1}$, $M_{2}$, and $M_{0}$.
Since $\phi_{P_{1}}|_{K}$ is a linear function, the integral of $\phi_{P_{1}}$ on $\triangle P_{1}M_{1}M_{2}$ equals $\phi_{P_{1}}(Q_{0})$ multiplied by the area of $\triangle P_{1}M_{1}M_{2}$. Thus,
\begin{eqnarray}
\int_{K_{0}} \phi_{P_{1}} \ud \ux
&=& \int_{z(Q_{0})}^{z(M_{0})}
    \left( \frac{z-z(Q_{0})}{z(M_{0})-z(Q_{0})}(\phi_{P_{1}}(M_{0})-\phi_{P_{1}}(Q_{0})) + \phi_{P_{1}}(Q_{0}) \right) \nonumber \\
&&  \qquad\quad \Big(\frac{z(M_{0})-z}{z(M_{0})-z(Q_{0})}\Big)^2 |\triangle P_{1}M_{1}M_{2}|  \,\ud z \nonumber \\
&=& \int_{0}^{\frac{1}{4}}
    \left( \frac{z}{\frac{1}{4}-0}
            ( \, (1-\frac{3}{4})-( 1-\frac{\frac{1}{2}+\frac{2}{3}}{3} )) + (1-\frac{\frac{1}{2}+\frac{2}{3}}{3})  \right)
         \Big(\frac{\frac{1}{4}-z}{\frac{1}{4}-0}\Big)^2 \frac{1}{6}  \,\ud z \nonumber \\
&=& \frac{25}{6912}.  \nonumber
\end{eqnarray}
From (\ref{eq:m1Vsd3D_a}), we obtain
\[
m_{1} = 36\int_{K_{0}} \phi_{P_{1}} \ud \ux = \frac{25}{192}.
\]

On the other hand,
\begin{eqnarray}
|K_{P_{i}}^{*}\cap K|=\frac{1}{4}|K| =\int_{K_{P_{i}}^{*}\cap K} \sum_{j=1}^{4} \phi_{P_{j}} \ud \ux = (m_{1}+3m_{2}) |K|.    \nonumber
\end{eqnarray}
Then,
\[
m_{2}=\frac{23}{576}.
\]

\vspace{10pt}

{\em The general $d$ case.} A similar procedure can be used in general $d$ dimensions. We have
\begin{eqnarray}
&& \int_{K_{P_{i}}^{*}\cap K} \phi_{P_{1}} \ud \ux
\nonumber \\
&=& d\int_{K_{0}} \phi_{P_{1}} \ud \ux       \nonumber\\
&=& d \int_{x_{d}(Q_{0})}^{x_{d}(M_{0})}
    \left( \frac{x_{d}-x_{d}(Q_{0})}{x_{d}(M_{0})-x_{d}(Q_{0})}(\phi_{P_{1}}(M_{0})-\phi_{P_{1}}(Q_{0})) + \phi_{P_{1}}(Q_{0}) \right) \times \nonumber \\
&&  \qquad\quad \Big(\frac{x_{d}(M_{0})-x_{d}}{x_{d}(M_{0})-x_{d}(Q_{0})}\Big)^{d-1} S_{K_{d-1}}  \,\ud x_{d} \nonumber \\
&=& d\int_{0}^{\frac{1}{d+1}}
    \left( \frac{x_{d}}{\frac{1}{d+1}-0}
            ( \, (1-\frac{d}{d+1})-( 1- \frac{\sum_{j=1}^{d}(1-\frac{1}{j})}{d} ))
            + (1-\frac{\sum_{j=1}^{d}(1-\frac{1}{j})}{d})  \right) \times       \nonumber\\
&&  \qquad\quad\Big(\frac{\frac{1}{d+1}-x_{d}}{\frac{1}{d+1}-0}\Big)^{d-1} |K|  \,\ud x_{d} \nonumber \\
&=& d\int_{0}^{\frac{1}{d+1}}
    \left( x_{d}(d+1)  \Big( \, \frac{1}{d+1} +  \frac{\sum_{j=1}^{d}\frac{1}{j}}{d} \Big) + \frac{\sum_{j=1}^{d}\frac{1}{j}}{d}  \right)
    \Big(\frac{\frac{1}{d+1}-x_{d}}{\frac{1}{d+1}-0}\Big)^{d-1} |K|  \,\ud x_{d} \nonumber \\
&=&  \frac{1}{(d+1)^3} \left( 1+(d+1)\sum_{i=1}^{d}\frac{1}{i} \right)|K|.
\nonumber
\end{eqnarray}
Here, $S_{K_{d-1}}$ is the $(d-1)$-dimensional measure of the face of $K_{P_{i}}^{*}\cap K$ restricted
on $x_{d}=0$, which is equal to $|K|$ in the current situation, and  $K_{0}$ denotes the polyhedron
bounded by the face of $K_{P_{i}}^{*}\cap K$ restricted on $x_{d}=0$ and $M_{0}$, whose $(d-1)$-dimensional
measure is $|K|/d$.
\begin{footnote}{For simplicity, in the $3$-dimensional case, $K_{0}$ denotes the tetrahedron
$P_{1}P_{2}P_{3}P_{4}$, i.e., an
half of the polyhedron, which is bounded by the face of $K_{P_{i}}^{*}\cap K$ restricted on $x_{d}=0$ and $M_{0}$.}
\end{footnote}
Thus,
\begin{equation}
m_1 = \frac{1}{(d+1)^3} \left( 1+(d+1)\sum_{i=1}^{d}\frac{1}{i} \right)
\label{eq:m1VsddD}
\end{equation}

Moreover, we have
\begin{eqnarray}
|K_{P_{i}}^{*}\cap K| = \frac{1}{d+1} |K|
=\int_{K_{P_{i}}^{*}\cap K} \sum_{j=1}^{d+1} \phi_{P_{j}} \ud \ux = (m_{1}+dm_{2}) |K|.    \nonumber
\end{eqnarray}
Combining this with (\ref{eq:m1VsddD}), we obtain (\ref{m1-m2}).



\end{document}